\documentclass[11pt, a4paper]{amsart}
\usepackage[T1]{fontenc}
\usepackage[utf8]{inputenc}
\usepackage{geometry}                
\usepackage{graphicx}
\usepackage{amsmath, amsthm, amssymb, amscd}
\usepackage{mathrsfs}
\usepackage{hyperref}
\usepackage[all]{xy}
\usepackage{tikz}
\usepackage{stmaryrd}

\usepackage{setspace}

\theoremstyle{plain}
\newtheorem{thm}{Theorem}[section]
\newtheorem{lem}[thm]{Lemma}
\newtheorem{prop}[thm]{Proposition}
\newtheorem{cor}[thm]{Corollary}

\newtheorem{conj}[thm]{Conjecture}

\theoremstyle{definition}
\newtheorem{defn}{Definition}[section]
\newtheorem{example}{Example}[section]

\theoremstyle{remark}
\newtheorem{rem}{Remark}[section]

\newcommand{\gl}{\mathrm{GL}}
\newcommand{\SL}{\mathrm{SL}}

\newcommand{\ks}{\mathfrak{S}}
\newcommand{\pgl}{\mathrm{PGL}}
\newcommand{\bbg}{\mathbb{G}}

\newcommand{\ggl}{\mathfrak{gl}}

\newcommand{\cf}{\mathcal{F}}
\newcommand{\co}{\mathcal{O}}
\newcommand{\ck}{\mathcal{K}}
\newcommand{\cl}{\mathcal{L}}

\newcommand{\cp}{\mathcal{P}}

\newcommand{\dd}{\mathscr{D}}
\newcommand{\xx}{\mathscr{X}}

\newcommand{\ka}{\mathfrak{a}}
\newcommand{\kg}{\mathfrak{g}}
\newcommand{\kn}{\mathfrak{n}}
\newcommand{\kp}{\mathfrak{p}}
\newcommand{\kt}{\mathfrak{t}}
\newcommand{\kb}{\mathfrak{b}}

\newcommand{\km}{\mathfrak{m}}

\newcommand{\kq}{\mathfrak{q}}

\newcommand{\kkk}{\mathfrak{K}}

\newcommand{\bn}{\mathbf{N}}
\newcommand{\bz}{\mathbf{Z}}
\newcommand{\br}{\mathbf{R}}
\newcommand{\bp}{\mathbf{P}}
\newcommand{\bq}{\mathbf{Q}}

\newcommand{\bc}{\mathbf{C}}

\newcommand{\bs}{\mathbf{S}}

\newcommand{\bnn}{\mathbf{n}}

\newcommand{\spec}{\mathrm{Spec}}

\newcommand{\val}{\mathrm{val}}
\newcommand{\lie}{\mathrm{Lie}}
\newcommand{\Ad}{\mathrm{Ad}}
\newcommand{\ad}{\mathrm{ad}}

\newcommand{\diag}{\mathrm{diag}}
\newcommand{\gal}{\mathrm{Gal}}

\newcommand{\sch}{\mathrm{Sch}}
\newcommand{\Hom}{\mathrm{Hom}}
\newcommand{\ec}{\mathrm{Ec}}

\newcommand{\rk}{\mathrm{rk}}
\newcommand{\reg}{\mathrm{reg}}
\newcommand{\sym}{\mathrm{Sym}}
\newcommand{\qlbar}{\overline{\mathbf{Q}}_{l}}
\newcommand{\Div}{\mathrm{Div}}
\newcommand{\Cl}{\mathrm{Cl}}
\newcommand{\ch}{\mathrm{ch}}
\newcommand{\Td}{\mathrm{Td}}

\newcommand{\ep}{\epsilon}

\newcommand{\wnn}{\widetilde{N}}
\newcommand{\wht}{\widehat{T}}
\newcommand{\wn}{\tilde{\mathfrak{n}}}

\newcommand{\tp}{\widetilde{\mathbf{P}}}

\author{Zongbin \textsc{Chen}}

\address{EPFL SB Mathgeom/Geom, MA B1 447, Station 8, CH-1015, Lausanne, Switzerland} 
\email{zongbin.chen@gmail.com} 
\title{On the fundamental domain of affine Springer fibers}

\begin{document}

\maketitle



\begin{abstract}
Let $G$ be a connected reductive algebraic group over an algebraically closed field $k$, $\gamma\in \kg(k(\!(\ep)\!))$ a semisimple regular element, we introduce a fundamental domain $F_{\gamma}$ for the affine Springer fibers $\xx_{\gamma}$. 
We show that the purity conjecture of $\xx_{\gamma}$ is equivalent to that of $F_{\gamma}$ via the Arthur-Kottwitz reduction. 
We then concentrate on the unramified affine Springer fibers for the group $\gl_{d}$. It turns out that their fundamental domains behave nicely with respect to the root valuation of $\gamma$. We formulate a rationality conjecture about a generating series of their Poincaré polynomials, and study them in detail for the group $\gl_{3}$. In particular, we pave them in affine spaces and we prove the rationality conjecture.

\end{abstract}

\section{Introduction}

Let $k$ be an algebraically closed field. Let $F=k(\!(\ep)\!)$ be the field of Laurent series with coefficients in $k$, $\co=k[\![\ep]\!]$ the ring of integers of $F$, $\kp=\ep k[\![\ep]\!]$ the maximal ideal of $\co$. We fix a separable algebraic closure $\overline{F}$ of $F$, let $\val:\overline{F}^{\times}\to \bq$ be the discrete valuation normalised by $\val(\ep)=1$. 

Let $G$ be a connected reductive algebraic group over $k$, we make the assumption that $\mathrm{char}(k)>\rk(G)$, where $\rk(G)$ is the semisimple rank of $G$. Let $G_{F}$ be the base change of $G$ from $k$ to $F$. Let $T$ be a maximal torus of $G_{F}$ over $F$. Their Lie algebras will be denoted by the corresponding Gothic letters. Let $K=G(\co)$ be the standard maximal compact subgroup of $G(F)$. We have the affine grassmannian $\xx=G(F)/K$, which is an ind-$k$-scheme. For a regular element $\gamma\in \kt(\co)$, the affine Springer fiber $\xx_{\gamma}$ at $\gamma$
$$
\xx_{\gamma}=\{g\in G(F)/K\,\mid\,\Ad(g^{-1})\gamma \in \kg(\co)\}
$$ 
was introduced by Kazhdan and Lusztig \cite{kl}. The most striking property that these affine Springer fibers are conjectured to have is the following:
 
\begin{conj}[Goresky-Kottwitz-MacPherson]\label{gkmconj1}
The cohomology of $\xx_{\gamma}$ is pure in the sense of Deligne.\end{conj}

Assuming this conjecture, Goresky, Kottwitz and MacPherson \cite{gkm1} have proved the fundamental lemma of Langlands-Shelstad in the unramified case, i.e. when the torus $T$ splits under an unramified extension of $F$. Following the same strategy, but assuming a truncated variant of the purity conjecture, Chaudouard and Laumon \cite{cl} also prove Arthur's weighted fundamental lemma in the unramified case. As it will turn out, their variant of the purity conjecture is equivalent to the conjecture \ref{gkmconj1}, see remark \ref{clpurity}.  

Although the fundamental lemma has been proven by Ng\^o \cite{ngo} and the weighted fundamental lemma by Chaudouard and Laumon \cite{cl2}, \cite{cl3}, the purity conjecture remains open except in several particular cases. Goresky, Kottwitz and MacPherson \cite{gkm2} have proved it when $\gamma$ is \emph{equivalued}, i.e. the elements $\alpha(\gamma)\in \overline{F}$ have the same valuation for all the roots $\alpha$ of $G_{\overline{F}}$ with respect to $T_{\overline{F}}$. Lucarelli \cite{l} constructed an affine paving of $\xx_{\gamma}$ for the unramified elements $\gamma \in \ggl_{3}(F)$, without the equivalued condition. Generalising his method in a more conceptual way, we \cite{chen} construct affine pavings of $\xx_{\gamma}$ for the unramified elements $\gamma \in \ggl_{4}(F)$. As a side result, we also complete the case of $\gl_{3}$.

In general, it is expected that $\xx_{\gamma}$ admits a \emph{Hessenberg paving}. By the word ``Hessenberg paving'' of an ind-$k$-scheme $X$, we mean an exhaustive increasing filtration
$\emptyset\subsetneq X_{1}\subsetneq X_{2}\subsetneq \cdots $ of $X$ by closed complete subschemes $X_{i}$ of finite type over $k$
such that each successive difference $X_{i+1}\backslash X_{i}$ is a disjoint union of iterated affine space bundles over Hessenberg varieties. (For the definition of Hessenberg variety, we refer the reader to \cite{gkm2}, \S 2.) 
When $\gamma$ is unramified, we believe that $\xx_{\gamma}$ even admits an affine paving. In the special case when $G$ is of type $A$, it seems to us that $\xx_{\gamma}$ always admits affine pavings.

One of the difficulties to construct affine pavings is due to the fact that the affine Springer fibers are generally not of finite type. But their structure is not completely arbitrary either. In fact, they have a large symmetry group. The group $T(F)$ acts on $\xx_{\gamma}$ with one of its orbits being dense open in $\xx_{\gamma}$. So the free abelian discrete group $\Lambda=\pi_{0}(T(F))$ acts simply and transitively on the irreducible components of $\xx_{\gamma}$. It is desirable to use this symmetry to reduce the study of $\xx_{\gamma}$ to that of its irreducible components. But the condition of irreducibility is difficult to explore. Instead, we construct a fundamental domain $F_{\gamma}$ of $\xx_{\gamma}$ with respect to the action of $\Lambda$, which should be exactly one of the irreducible components of $\xx_{\gamma}$. 

When $T$ is split over $F$, the construction of $F_{\gamma}$ runs roughly as follows: Since $T$ splits, we have $T=T_{0, F}$ for some maximal torus $T_{0}$ of $G$ over $k$. Let $\cp(T_{0})$ be the set of Borel subgroups of $G$ containing $T_{0}$. For $x\in \xx,\, B\in \cp(T_{0})$, let $f_{B}(x)\in X_{*}(T_{0})$ be the unique co-character $\nu$ such that $x\in U_{B}(F)\ep^{\nu}K/K$, where $U_{B}$ is the unipotent radical of $B$. We denote by $\ec(x)$ the convex hull of $(f_{B}(x))_{B\in \cp(T_{0})}$ in $X_{*}(T_{0})\otimes \br$. Take a point $x_{0}$ in general position on $\xx_{\gamma}$, let  
$$
F_{\gamma}=\{x\in \xx_{\gamma}\mid \ec(x)\subset \ec(x_{0}),\,\nu_{G}(x)=\nu_{G}(x_{0})\},
$$
where the second condition in the bracket means that $x$ and $x_{0}$ lie on the same connected component of $\xx$.

Our first main result is the following:

\begin{thm}\label{main1}

For any $\gamma\in \kt(\co)$, suppose that $F_{\gamma}^{M}$ is cohomologically pure for any proper Levi subgroup $M$ of $G$ containing $T$. Then $\xx_{\gamma}$ is cohomologically pure if and only if $F_{\gamma}$ is. 
\end{thm}

For the group $G=\gl_{1}$, we have $F_{\gamma}=\mathrm{pt}$ et $\xx_{\gamma}=\bz \times \mathrm{pt}$, which are obviously cohomologically pure. Using theorem \ref{main1} inductively, we see that the conjecture of Goresky, Kottwitz and MacPherson is equivalent to

\begin{conj}\label{gkmconj2}
The cohomology of $F_{\gamma}$ is pure in the sense of Deligne.
\end{conj}

At this stage, we should make a few comments on the advantage of conjecture \ref{gkmconj2} over conjecture \ref{gkmconj1}. Firstly, since $F_{\gamma}$ is of finite type, we can count its rational points when the base field $k$ is taken to be a finite field. The counting result can give us a hint on how to construct affine pavings. Secondly, in a forthcoming paper (a very preliminary version can be found at \cite{chen2}), we outline a conjectural general procedure to construct affine pavings for cohomologically pure algebraic varieties admitting nice torus action.
Basically, we look at the moment graph of the torus action and introduce a formal Poncar\'e polynomial for each acyclic orientation of the moment graph. We conjecture that whenever the formal Poincar\'e polynomial attains the minimal, we obtain a generalised affine paving. The reader is referred to \S3.4 of \cite{chen2} for more details. To apply this conjecture, we need the ind-scheme in question to be of finite type. In fact, it is this conjecture that motivates us to the construction of $F_{\gamma}$.

We believe that $F_{\gamma}$ is an irreducible component of $\xx_{\gamma}$, and that it is also the normalisation of $\Lambda\backslash\xx_{\gamma}$. This picture can be put in another context which is better suited for deformation. Recall the following construction of Laumon \cite{lau}: Let $C$ be a rational projective curve with a unique planar singularity at $x\in C$ such that the completed local ring $\widehat{\co}_{C,x}$ is isomorphic to $\co[\gamma]$. Let $\overline{\mathrm{Jac}}_{C}$ be the compactified Jacobian of $C$, which is the moduli space of torsion-free coherent sheaves of generic rank $1$ and of degree $0$ on $C$. Among others, Laumon shows that there exists a morphism $\Lambda\backslash\xx_{\gamma} \to \overline{\mathrm{Jac}}_{C}$, which is finite, \emph{radicial} and surjective. In particular, this implies that the \'etale cohomologies of $\Lambda\backslash\xx_{\gamma}$ and $ \overline{\mathrm{Jac}}_{C}$ are isomorphic. In general, we have the ``formule de produit'' de Ng\^o \cite{ngo}, \S 4.15, which gives a uniformisation of the compactified Jacobian of a projective irreducible algebraic curve with planar singularities by products of the affine Springer fibers associated with the singularities.  
Based on these observations, we restate the conjectures of Goresky, Kottwitz, MacPherson and of Laumon \cite{lau} \S 3.2 as follows, which hopefully may lead to a proof of the purity conjecture by deformation.

\begin{conj}
Let $C$ be a projective geometrically integral algebraic curve over $k$. Suppose that all the singularities of $C$ are planar. Then the  normalisation of the compactified Jacobian of $C$ is cohomologically pure.
\end{conj}

Now we restrict to the group $G=\gl_{d+1}$. Let $T$ be the maximal torus of $G$ of diagonal matrices, let $B_{0}$ be the Borel subgroup of $G$ of the upper triangular matrices. Let $\Phi=\{\alpha_{i,j}\}$ be the root system of $G$ with respect to $T$, let $\alpha_{i}=\alpha_{i,i+1},\,i=1,\cdots,d,$ be the simple roots with respect to $B_{0}$. Let $\gamma\in \kt(\co)$ be regular, it is said to be in minimal form if 
$$
\val(\alpha_{i,j}(\gamma))=\min_{i\leq l\leq j-1}\{\val(\alpha_{l}(\gamma))\}, \quad \forall\, i<j.
$$
In this case, we say that the root valuation of $\gamma$ is $(\val(\alpha_{1}(\gamma)),\cdots,\val(\alpha_{d}(\gamma)))$. According to \cite{chen} appendix, we can always conjugate $\gamma$ such that it is in minimal form, and for $\bnn=(n_{1},\cdots,n_{d})\in \bn^{d}$, we can find $\gamma\in \kt(\co)$ in minimal form with root valuation $\bnn$.  

\begin{conj}\label{rationality}

The topology of $F_{\gamma}$ only depends on its root valuation. Let $P_{\bnn}(q)$ be its Poincaré polynomial. The power series
$$
Q(q;\vec{t}\;):=\sum_{n_{1}=1}^{+\infty}\cdots \sum_{n_{d}=1}^{+\infty}  P_{(n_{1},\cdots,n_{d})}(q)\, t_{1}^{n_{1}}\cdots t_{d}^{n_{d}}\in \bz [[ q; t_{1},\cdots,t_{d}]]
$$
is a rational fraction, i.e. it is an element of $\bz(q; t_{1},\cdots,t_{d})$. 
\end{conj}

In particular, the conjecture implies that it is enough to do finitely many computations in order to get the Poincar\'e polynomials of all the $F_{\gamma}$'s.

Chaudouard and Laumon \cite{cl} have calculated the $T$-equivariant homology of cohomologically pure truncated affine Springer fibers, following the strategy of Goresky, Kottwitz and MacPherson \cite{gkm1}. Assume that $F_{\gamma}$ is cohomologically pure, we then reduce the rationality of $Q(q; \vec{t}\;)$ to that of another power series, which admits a certain geometric interpretation via toric varieties. We refer the reader to \S 4 for more details.

For the groups $\gl_{2}$ and $\gl_{3}$, we have been able to calculate the Poincar\'e polynomial of $F_{\gamma}$, without any hypothesis. The same method works for $\gl_{4}$, but the combinatorics is too complicated to write down.

\begin{thm}\label{main2}

\begin{enumerate}

\item
For $G=\gl_{2}$, any element $\gamma\in \kt(\co)$ is automatically in minimal form. Let $n$ be its root valuation. Then the Poincar\'e polynomial of $F_{\gamma}$ is $\sum_{i=0}^{n}q^{i}$.

\item
For $G=\gl_{3}$, let $\bnn=(n_{1},n_{2})\in \bn^{2},\, n_{1}\leq n_{2}$, let $\gamma\in \kt(\co)$ be in minimal form with root valuation $\bnn$. The fundamental domain $F_{\gamma}$ can be paved in affine spaces, and the paving only depends on $\bnn$. Its Poincaré polynomial is 
\begin{eqnarray*}
P_{\bnn}(q)&=&\sum_{i=1}^{n_{1}}i(q^{4i-2}+q^{4i-4})+\sum_{i=2n_{1}}^{n_{1}+n_{2}-1}(2n_{1}+1)q^{2i}\\
&&+\sum_{i=n_{1}+n_{2}}^{2n_{1}+n_{2}-1}4(2n_{1}+n_{2}-i)q^{2i}+q^{4n_{1}+2n_{2}}.
\end{eqnarray*}
\end{enumerate}
\end{thm}

The rationality conjecture in these cases are easy consequences of the theorem.

\subsection*{Notations}

We fix a split maximal torus $A$ of $G$ over $k$. Let $\Phi=\Phi(G,A)$ be the root system of $G$ with respect to $A$, let $W$ be the Weyl group of $G$ with respect to $A$. For any subgroup $H$ of $G$ which is stable under the conjugation of $A$, we note $\Phi(H,A)$ for the roots appearing in $\lie(H)$. We fix a Borel subgroup $B_{0}$ of $G$ containing $A$. Let $\Delta$ be the set of simple roots with respect to $B_{0}$, let $(\varpi_{\alpha})_{\alpha\in \Delta}$ be the corresponding fundamental weights. To an element $\alpha\in \Delta$, we have a unique maximal parabolic subgroup $P_{\alpha}$ of $G$ containing $B_{0}$ such that $\Phi(N_{P_{\alpha}},A)\cap \Delta=\alpha$, where $N_{P_{\alpha}}$ is the unipotent radical of $P_{\alpha}$. This gives a bijective correspondence between the simple roots in $\Delta$ and the maximal parabolic subgroups of $G$ containing $B_{0}$. Any maximal parabolic subgroup $P$ of $G$ is conjugate to certain $P_{\alpha}$ by an element $w\in W$, the element $w\varpi_{\alpha}$ doesn't depend on the choice of $w$, we denote it by $\varpi_{P}$.

We use the $(G,M)$ notation of Arthur. Let $\cf(A)$ be the set of parabolic subgroups of $G$ containing $A$, let $\cl(A)$ be the set of Levi subgroups of $G$ containing $A$. For every $M\in \cl(A)$, we denote by $\cp(M)$ the set of parabolic subgroups of $G$ whose Levi factor is $M$, and by $\cf(M)$ the set of parabolic subgroups of $G$ containing $M$. For $P\in \cp(M)$, we denote by $P^{-}$ the opposite of $P$ with respect to $M$. Let $X^*(M)=\Hom(M, \bbg_m)$ and $\ka_M^{*}=X^*(M)\otimes\br$. The restriction $X^{*}(M)\to X^{*}(A)$ induces an injection $\ka_{M}^{*}\hookrightarrow \ka_{A}^{*}$. Let $(\ka_{A}^{M})^{*}$ be the subspace of $\ka_{A}^{*}$ generated by $\Phi(M,A)$. We have the decomposition in direct sums
$$
\ka_{A}^{*}=(\ka_{A}^{M})^{*}\oplus \ka_{M}^{*}.
$$

The canonical pairing 
$$
X_{*}(A)\times X^{*}(A)\to \bz
$$ 
can be extended bilinearly to $\ka_{A}\times \ka_{A}^{*}\to \br$, with $\ka_{A}=X_{*}(A)\otimes \br$. For $M\in \cl(A)$, let $\ka_{A}^{M}\subset \ka_{A}$ be the subspace orthogonal to $\ka_{M}^{*}$, and $\ka_{M} \subset \ka_{A}$ be the subspace orthogonal to $(\ka_{A}^{M})^{*}$. We have also dually the decomposition
$$
\ka_{A}=\ka_{M}\oplus \ka_{A}^{M},
$$
let $\pi_{M},\,\pi^{M}$ be the projections to the two factors. More generally, for $L,M\in \cf(A),\,M\subset L$, we also have a decomposition
$$
\ka_{M}=\ka_{L}\oplus \ka_{M}^{L}.
$$
To save notation, we also write $\pi_{L},\,\pi^{L}$ for the projections to the two factors.

We identify $X_{*}(A)$ with $A(F)/A(\co)$ by sending $\chi$ to $\chi(\ep)$. With this identification, the canonical surjection $A(F)\to A(F)/A(\co)$ can be viewed as
\begin{equation}\label{indexT}
A(F)\to X_{*}(A).
\end{equation}

We use $\Lambda_{G}$ to denote the quotient of $X_{*}(A)$ by the coroot lattice of $G$ (the subgroup of $X_{*}(A)$ generated by the coroots of $A$ in $G$). We have a canonical homomorphism
\begin{equation}\label{indexM}
G(F)\to \Lambda_{G},
\end{equation}
which is characterised by the following properties: it is trivial on the image of $G_{\mathrm{sc}}(F)$ in $G(F)$ ($G_{\mathrm{sc}}$ is the simply connected cover of the derived group of $G$), and its restriction to $A(F)$ coincides with the composition of (\ref{indexT}) with the projection of $X_{*}(A)$ to $\Lambda_{G}$. Since the morphism (\ref{indexM}) is trivial on $G(\co)$, it descends to a map
$$
\nu_{G}:\xx\to \Lambda_{G},
$$
whose fibers are the connected components of $\xx$.

Finally, we suppose that $\gamma\in \kt(\co)$ satisfies $\gamma\equiv 0 \mod \ep$ to avoid unnecessary complications.

\section{The fundamental domain}

\subsection{Truncated affine Springer fibers}

For $M\in \cl(A)$, the natural inclusion of $M(F)$ in $G(F)$ induces a closed immersion of $\xx^{M}$ in $\xx^{G}$. For $P=MN\in \cf(A)$, we have the retraction
$$
f_{P}:\xx\to \xx^{M}
$$ 
which sends $gK=nmK$ to $mM(\co)$, where $g=nmk,\,n\in N(F),\, m\in M(F),\, k\in K$ is the Iwasawa decomposition.

\begin{rem}
We want to emphasise that the retraction $f_{P}$ is not a morphism between ind-$k$-schemes. In fact, it is not even a continuous map. But it becomes a morphism when restricted to the inverse image of each connected component of $\xx^{M}$. 

To see this, for $\mu\in \Lambda_{M}$, let $\xx^{M,\mu}=\nu_{M}^{-1}(\mu)$, the inverse image $f_{P}^{-1}(\xx^{M,\mu})=N(F)\xx^{M,\mu}$ is a locally closed ind-$k$-scheme of $\xx$. The restriction  
$$
f_{P}:N(F)\xx^{M,\mu}\to \xx^{M,\mu}
$$
is easily seen to be an infinite dimensional homogeneous affine fibration. The problem with the global $f_{P}$ is that, while $\xx^{M}$ is the disjoint union of its connected components $\xx^{M,\mu}$, the affine grassmannian $\xx$ is not the disjoint union of $N(F)\xx^{M,\mu}$ as an ind-$k$-scheme. 

\end{rem}

More generally we can define $f^{L}_{P_{L}}:\xx^{L}\to \xx^{M}$ for $L\in \cl(A), \, L\supset M$ and $P_{L}\in \cp^{L}(M)$. These retractions satisfy the transitivity property: Suppose that $Q\in \cp(L)$ satisfies $Q \supset P$, then
$$
f_{P}=f^{L}_{P\cap L}\circ f_{Q}.
$$

For $P\in \cf(A)$, we have the function $H_{P}:\xx\to \ka_{M}^{G}$ which is the composition 
$$
H_{P}:\xx\xrightarrow{f_{P}}\xx^{M}\xrightarrow{\nu_{M}}\Lambda_{M}\to \ka_{M}^G.
$$

There is a notion of adjacency among the parabolic subgroups in $\cp(M)$: Two parabolic subgroups $P_{1}=MN_{1},\,P_{2}=MN_{2}\in \cp(M)$ are said to be \emph{adjacent} if both of them are contained in a parabolic subgroup $Q=LU$ such that $L\supset M$ and $\rk(L)=\rk(M)+1$. Given such an adjacent pair, we define an element $\beta_{P_{1},P_{2}}\in \Lambda_{M}$ in the following way: Consider the collection of elements in $\Lambda_{M}$ obtained from coroots of $A$ in $\kn_{1}\cap \kn_{2}^{-}$, we define $\beta_{P_{1},P_{2}}$ to be the minimal element in this collection, i.e. all the other elements are positive integral multiples of it. Note that $\beta_{P_{2},P_{1}}=-\beta_{P_{1},P_{2}}$, and if $M=A$, then $\beta_{P_{1},P_{2}}$ is the unique coroot which is positive for $P_{1}$ and negative for $P_{2}$.

\begin{prop}[Arthur\cite{a}]

Let $P_{1},\,P_{2}\in \cp(M)$ be two adjacent parabolic subgroups. For any $x\in \xx$, we have
$$
H_{P_{1}}(x)-H_{P_{2}}(x)=n(x,P_{1},P_{2})\cdot \beta_{P_{1},P_{2}},
$$
with $n(x, P_{1}, P_{2})\in \bz_{\geq 0}$.

\end{prop}

\begin{proof}

We give a proof for the case when $M=A$ is a split maximal torus of $G$, the general case follows by applying the projection from $\ka_{A}^{G}$ to $\ka_{M}^{G}$.

For any two adjacent Borel subgroups $B',B''\in \cp(A)$, let $P$ be the parabolic subgroup generated by $B'$ and $B''$. Let $P=LU$ be the Levi factorisation. The application $H_{B'}$ factor through $f_{P}$, i.e. we have commutative diagram
$$
\xymatrix{
\xx \ar[d]_{f_{P}} \ar[dr]^{H_{B'}} &\\
\xx^{L}\ar[r]_{H^{L}_{B'\cap L}}&\ka_{T}^{G}}
$$
and similarly for $H_{B''}$. Since $L$ has semisimple rank $1$, the proposition is thus reduced to $G=\mathrm{SL}_{2}$. In this case, let $A$ be the maximal torus of the diagonal matrices, $B'=\begin{pmatrix}*&*\\ & * 
\end{pmatrix}
$,
$
B''=\begin{pmatrix}*&\\ * & * 
\end{pmatrix}$, and we identify $\ka_{A}^{G}$ with the line $H=\{(x,-x)\mid x\in \br\}\subset \br^{2}$ in the usual way. By the Iwasawa decomposition, any point $x\in \xx$ can be written as
$
x=\begin{pmatrix}a&b \\ &d\end{pmatrix}K.
$ 
Let $m=\min\{\val(a),\,\val(b)\}$, $n=\val(d)$, then $m+n\leq \val(a)+\val(d)=0$ and
$$
H_{B'}(x)=(-n,n), \quad H_{B''}(x)=(m,-m).
$$
So 
$$
H_{B'}(x)-H_{B''}(x)=(-(n+m), n+m)=-(n+m)\cdot \beta_{B',B''},
$$
and the proposition follows.
\end{proof}

For any point $x\in \xx$, we write $\ec_{M}(x)$ for the convex hull in $\ka_{M}^{G}$ of the $H_{P}(x),\,P\in \cp(M)$. For any $Q\in \cf(M)$, we denote by $\ec^{Q}_{M}(x)$ the face of $\ec_{M}(x)$ whose vertices are $H_{P}(x),\,P\in \cp(M),\,P\subset Q$. When $M=A$, we simplify the notations to $\ec(x)$ and $\ec^{Q}(x)$ respectively.

\begin{defn}
A family $D=(\lambda_{P})_{P\in \cp(M)}$ of elements in $\ka_{M}^{G}$ is called a \emph{positive} $(G,M)$-\emph{orthogonal family} if it satisfies
$$
\lambda_{P_{1}}-\lambda_{P_{2}}\in \br_{\geq 0}\cdot \beta_{P_{1},P_{2}},
$$
for any two adjacent parabolic subgroups $P_{1},\, P_{2}\in \cp(M)$.
\end{defn}

Given such a positive $(G,M)$-orthogonal family, we will denote again by $D$ the convex hull of the $\lambda_{P}$'s. For $Q=LU\in \cf(M)$, parallel to $\ec_{M}^{Q}(x)$, we denote by $D^{Q}$ the face of $D$ whose vertices are $\lambda_{P},\,P\in \cp(M),\,P\subset Q$. With the projection $\pi^{L}$, it will also be seen as a positive $(L,M)$-orthogonal family.

Following Chaudouard and Laumon \cite{cl}, we define the \emph{truncated affine grassmannian} $\xx(D)$ to be
$$
\xx(D)=\{x\in \xx\mid \ec_{M}(x)\subset D\},
$$
and the \emph{truncated affine Springer fiber} $\xx_{\gamma}(D)$ to be the intersection $\xx_{\gamma}\cap \xx(D)$. It should be pointed out that both $\xx(D)$ and $\xx_{\gamma}(D)$ can have several connected components, and there is slight difference between different components. We give an illustration of this point in the coming example \ref{diffcomp}.

\begin{example}\label{tschubert}

Let $\lambda\in X_{*}(A)$ be a dominant cocharacter with respect to $B_{0}$, we have the positive $(G,A)$-orthogonal family $D=(\lambda_{B})_{B\in \cp(T)}$ with $\lambda_{wB_{0}}=w (\lambda)$. Let $\xx^{|\lambda|}(D)$ be the connected component of $\xx(D)$ containing $\epsilon^{\lambda}$, then 
$$
\xx^{|\lambda|}(D)=\sch(\lambda),
$$
where $\sch(\lambda)$ is the affine Schubert variety $\overline{K\epsilon^{\lambda}K/K}$.

To see this, we need the Bruhat-Tits decomposition of $\sch(\lambda)$. Let $I$ be the standard Iwahori subgroup, i.e. it is the pre-image of $B_{0}$ under the reduction $G(\co)\xrightarrow{\mod \ep}G(k)$. The Bruhat-Tits decomposition states that
$$
\sch(\lambda)=\bigcup_{\stackrel{\mu\in X_{*}(A)}{\mu\prec\lambda}}I\ep^{\mu}K/K,
$$
where $\prec$ means the Bruhat-Tits order on $X_{*}(A)$ with respect to $I$.

Now that $\sch(\mu)$ is an $A$-invariant projective algebraic variety, we see that 
$$
\lim_{t\to 0}\chi(t)x\in \sch(\lambda)^{A},
$$
for any point $x\in \sch(\lambda)$ and any regular cocharacter $\chi\in X_{*}(A)$. This implies that $\sch(\lambda)\subset \xx^{|\lambda|}(D)$.

Conversely, if there exists any point $x$ in $\xx^{|\lambda|}(D)\backslash \sch(\lambda)$, it must lie in $I\ep^{\nu}K/K$ for some $\ep^{\nu}\notin \sch(\lambda)^{A}$, i.e. the image of $\nu$ in $\ka_{A}^{G}$ will lie outside the convex polytope $D$. Now look at the affine Schubert cell $I\ep^{\nu}K/K$. Choose an element $a\in \kt$ such that the associated Moy-Prasad filtration gives $(G_{F})_{a,0}=I$. More precisely, the choice of $a$ satisfies
$$
\lie(I)=\bigoplus_{\stackrel{(\alpha,n)\in \Phi(G,T)\times \bz}{\alpha(a)+n\geq 0}}\kg_{\alpha}\ep^{n}+ \kg\ep^{N}, \quad \forall N\gg 0.
$$
The reader can refer to \cite{chen}, \S2 for a brief review of Moy-Prasad filtration. We have 
\begin{equation}\label{mpiso}
I\ep^{\nu}K/K\cong I/I\cap \Ad(\ep^{\nu})K\cong \bigoplus_{\stackrel{(\alpha,n)\in \Phi(G,T)\times \bz}{\alpha(a)+n\geq 0,\,\alpha(\nu)>n}}\kg_{\alpha}\ep^{n}.
\end{equation}
Let $U$ be the unipotent subgroup of $G$ with Lie algebra 
$$
\bigoplus_{\stackrel{\alpha\in \Phi(G,T)}{\alpha(\nu+a)>0}}\kg_{\alpha}.
$$
It follows from the isomorphism $(\ref{mpiso})$ that 
$$
I\ep^{\nu}K/K\subset U(\co)\ep^{\nu}K/K.
$$
Let $B\in \cp(A)$ be a Borel subgroup containing $U$. The above inclusion implies $H_{B}(x)=\nu\notin D$, contradictory to the hypothesis $x\in \xx^{|\lambda|}(D)$.


\end{example}

\begin{example}\label{diffcomp}

Let $G=\gl_{d+1}$, let $A$ be the maximal torus of the diagonal matrices. Let $D\subset \ka_{A}^{G}$ be a positive $(G,A)$-orthogonal family, we want to look at the difference between connected components of $\xx(D)$.

The map $\nu_{G}:\xx\to \Lambda_{G}\cong\bz$ sends $gK$ to $\val(\det(g))$ for any $g\in G(F)$. Denote $\nu_{G}^{-1}(n)$ by $\xx^{(n)}$. We can identify the central connected component $\xx^{(0)}$ with $\xx^{\SL_{d+1}}$. Let $\Pi_{i}=\diag(\ep,\cdots, \ep, 1,\cdots,1)$ with $i$ terms of $\ep$, $i=1, \cdots,d$. The translation by $\ep^{n}\Pi_{i}$ gives an isomorphism between $\xx^{(0)}$ and $\xx^{(n(d+1)+i)}$. 

Since $H_{B}(\ep^{n}x)=H_{B}(x)$ for any $x\in \xx,\,n\in \bz$, we only need to look at the differences between $\xx^{(0)}(D)$ and $\xx^{(i)}(D)$, for $i=1,\cdots, d$. Let $B_{0}$ be the Borel subgroup of the upper triangular matrices. Let $\{\alpha_{i}\}_{i=1}^{d}$ be the simple roots of $G$ with respect to $B_{0}$. Let $\{\varpi_{i}^{\vee}\}_{i=1}^{d}$ be the corresponding fundamental coweights of $G$, i.e. they are elements in $X_{*}(A)\otimes \bq$ characterised by $\alpha_{i}(\varpi_{j}^{\vee})=\delta_{ij},\,\forall\, i,j=1,\cdots, d.$ Since the translation by $\Pi_{i}$ induces an isomorphism between $\xx^{(0)}$ and $\xx^{(i)}$, the image of $A$-invariant points $\xx^{(0),A}$ and $\xx^{(i),A}$ in $\ka_{A}^{G}$ will differ by a translation of $\varpi_{i}^{\vee}$. Hence a translation by $\Pi_{i}$ is necessary to get an isomorphism between $\xx^{(0)}(D)$ and $\xx^{(i)}(D+\varpi_{i}^{\vee})$, where $D+\varpi_{i}^{\vee}$ is the translation of $D$ by $\varpi_{i}^{\vee}$.

\end{example}

\subsection{The fundamental domain}

We begin by recalling several results concerning the action of $T(F)$ on the affine Springer fiber $\xx_{\gamma}$. Let $\Lambda=\pi_{0}(T(F))$, it is a discrete free abelian group.

\begin{prop}[Kazhdan-Lusztig \cite{kl}]\label{klfinite}

The group $\Lambda$ acts freely on $\xx_{\gamma}$ with the quotient $\Lambda\backslash \xx_{\gamma}$ being a projective $k$-scheme, and the quotient map $\xx\to \Lambda\backslash \xx_{\gamma}$ is an \'etale Galois covering.

\end{prop}

A point $x=gK\in \xx_{\gamma}$ is said to be \emph{regular} if the image of $\Ad(g^{-1})\gamma$ under the reduction $\kg(\co)\to \kg(k)$ is regular. Let $\xx_{\gamma}^{\reg}$ be the open sub variety of $\xx_{\gamma}$ of the regular points.

\begin{prop}[Bezrukavnikov \cite{b}]\label{bezru}
The group $T(F)$ acts transitively on $\xx_{\gamma}^{\reg}$.
\end{prop}

\begin{prop}[Ng\^o \cite{ngo}]

The subvariety $\xx_{\gamma}^{\reg}$ is open dense in $\xx_{\gamma}$. 

\end{prop}

The last proposition is proved in an indirect way. In fact, one needs to use Laumon's observation on the affine Springer fibers and the compactified Jacobians, as recalled briefly in the introduction, and to use a corresponding property of the compactified Jacobians. As a consequence of the above two propositions, the abelian group $\Lambda$ acts freely and transitively on the irreducible components of $\xx_{\gamma}$.

Let $S$ be the maximal $F$-split subtorus of $T$. Let $M_{0}$ be the connected component of the centraliser of $S$ in $G$, then $T$ is anisotropic modulo center in $M_{0,F}$. We also have $\Lambda=\Lambda_{M_{0}}$. Without any loss of generality, we may assume that $M_{0}$ contains $A$.

Goresky, Kottwitz and MacPherson \cite{gkm3} have given a characterisation of the regular points in $\xx_{\gamma}$. To formulate it, we need to define an invariant $n(\gamma,P_{1},P_{2})\in \bz_{\geq 0}$ for any two adjacent parabolic subgroups $P_{1}=M_{0}N_{1},P_{2}=M_{0}N_{2}\in \cp(M_{0})$. The Galois group $\gal(\overline{F}/F)$ acts on the set of roots of $T_{\overline{F}}$ in $\kn_{1}\cap \kn_{2}^{-}$. Let $\alpha$ be such a root, let $F_{\alpha}$ be the field of definition of $\alpha$. Let $\val_{F_{\alpha}}$ be the valuation normalised such that any uniformiser in $F_{\alpha}$ has valuation $1$, i.e. $\val_{F_{\alpha}}(\ep)=[F_{\alpha}:F]$. Let $m_{\alpha}$ be the unique positive integer such that the image of $\alpha^{\vee}$ in $\Lambda_{M_{0}}$ is equal to $m_{\alpha}\cdot \beta_{P_{1},P_{2}}$. Now we define
$$
n(\gamma, P_{1},P_{2})=\sum \val_{F_{\alpha}}(\alpha(\gamma))\cdot m_{\alpha},
$$
where the sum is taken over a set of representatives $\alpha$ of the orbits of $\gal(\overline{F}/F)$ on the set of roots of  $T_{\overline{F}}$ in $\kn_{1}\cap \kn_{2}^{-}$.

\begin{prop}[Goresky-Kottwitz-MacPherson]\label{gkmreg}

Let $x\in \xx_{\gamma}$. 

\begin{enumerate}

\item For any two adjacent parabolic subgroups $P_{1}, P_{2}\in \cp(M_{0})$, we have
$$
n(x,P_{1},P_{2})\leq n(\gamma,P_{1},P_{2}).
$$

\item The point $x$ is regular in $\xx_{\gamma}$ if and only if the following two conditions holds:
\begin{enumerate}
\item
the point $f_{P}(x)$ is regular in $\xx^{M_{0}}_{\gamma}$ for all $P\in \cp(M_{0})$;

\item
for any two adjacent parabolic subgroups $P_{1}, P_{2}$ in $\cp(M_{0})$, one has
$$
n(x,P_{1},P_{2})=n(\gamma,P_{1},P_{2}).
$$

\end{enumerate}
\end{enumerate}
\end{prop}

In the proof of Goresky, Kottwitz and MacPherson, the general case is deduced from the unramified case, by base change to the splitting field of $\gamma$. We will reproduce their proof in the unramified case.

\begin{lem}[Goresky-Kottwitz-MacPherson]\label{regularreduction}

Let $\gamma\in \ka(\co)$. A point $x\in \xx_{\gamma}$ is regular if and only if for any Levi subgroup $M\in \cl(A)$ of semisimple rank $1$, the point $f_{P}(x)\in \xx_{\gamma}^{M}$ is regular for any $P\in \cp(M)$.

\end{lem}

\begin{proof}

For $x=gK\in \xx_{\gamma}$, the image of $\Ad(g)^{-1}\gamma$ under the reduction $\kg(\co)\to \kg(k)$ is well defined up to conjugacy, we denote it by $u_{G}(x)$. For any $P=MN\in \cf(A)$, let $g=pk,\,p\in P(F),\, k\in K$, then $\Ad(p)^{-1}\gamma\in \kp(F)\cap \kg(\co)=\kp(\co)$. Its image in $\kp(k)$ under the reduction is well defined up to conjugacy, we will denote it by $u_{P}(x)$. It is obvious that $u_{P}(x)$ goes to $u_{M}(f_{P}(x))$ under the projection $\kp\to \km$. So if $u_{G}(x)$ is regular, then $u_{M}(f_{P}(x))$ is regular since $u_{P}(x)$ lies in the same conjugacy class as $u_{G}(x)$ in $\kg(k)$. This proves the necessary part of the lemma.

For sufficiency, it is enough to prove that $u_{B_{0}}(x)$ is regular. For $\alpha\in \Delta$, let $Q_{\alpha}$ be the parabolic subgroup generated by $B_{0}$ and $s_{\alpha}\cdot B_{0}$, where $s_{\alpha}\in W$ is the simple reflection associated to $\alpha$. Let $Q_{\alpha}=M_{\alpha}N_{\alpha}$ be the Levi decomposition, then $M_{\alpha}$ is of semisimple rank $1$. Now $u_{B_{0}}(x)$ goes to $u_{M_{\alpha}}(f_{Q_{\alpha}}(x))$ under the composition $\kb_{0}\hookrightarrow \kq_{\alpha}\twoheadrightarrow \km_{\alpha}$. Since $u_{M_{\alpha}}(f_{Q_{\alpha}}(x))$ is regular in $\km_{\alpha}$ for any $\alpha\in \Delta$, this implies that $u_{B_{0}}(x)$ is regular.

\end{proof}

\begin{proof}[Proof of proposition \ref{gkmreg} when $M_{0}=A$.]

First of all, observe that for any $x,y\in \xx$ such that $y$ lies in the closure of the orbit $T(\co)\cdot x$, we have $\ec(y)\subset \ec(x)$. Now that $\xx_{\gamma}^{\reg}$ is dense open in $\xx_{\gamma}$, it suffices to prove the second assertion. By lemma \ref{regularreduction}, it suffices to prove the proposition for $G=\mathrm{GL}_{2}$. This follows from proposition \ref{classical example}, where we will pick a particular regular point $x_{0}\in \xx_{\gamma}^{\reg}$ and calculate that
$$
H_{B}(x_{0})=(\val(\alpha(\gamma)),\,0),\quad H_{B^{-}}(x_{0})=(0, \,\val(\alpha(\gamma))).
$$
It is obvious that $H_{B}(x_{0})-H_{B^{-}}(x_{0})=\val(\alpha(\gamma))\cdot \alpha^{\vee}$.

\end{proof}

The above results motivate the following definition.

\begin{defn}
Take a regular point $x_{0}\in \xx_{\gamma}^{\reg}$. Let 
$$
F_{\gamma}=\{x\in \xx_{\gamma}\mid \ec_{M_{0}}(x)\subset \ec_{M_{0}}(x_{0}), \,\nu_{G}(x)=\nu_{G}(x_{0})\}.
$$
We call it \emph{the fundamental domain} of $\xx_{\gamma}$. \end{defn}

It is clear that different choice of $x_{0}\in \xx_{\gamma}^{\reg}$ gives rise to isomorphic fundamental domain. It is also clear that $F_{\gamma}$ contains an irreducible component of $\xx_{\gamma}$, but it is more subtle whether they are isomorphic.

\begin{prop}

The fundamental domain $F_{\gamma}$ is a $k$-scheme of finite type. It is also the fundamental domain of $\xx_{\gamma}$ with respect to the action of $\Lambda$ in the usual sense, i.e. we have $\xx_{\gamma}=\Lambda\cdot F_{\gamma}$, and any two translations of $F_{\gamma}$ by elements of $\Lambda$ intersect in a closed sub variety of dimension strictly less than that of $F_{\gamma}$. 

\end{prop}

\begin{proof}

Since $\gamma$ is anisotropic modulo center in $M_{0}(F)$, the connected components of $\xx_{\gamma}^{M_{0}}$ are projective $k$-schemes by proposition \ref{klfinite}. This implies that there exists a bounded convex polytope $\Sigma_{1}$ in $\ka_{A}^{M_{0}}$ such that $\pi^{M_{0}}(\ec^{P}(x))=\ec^{(M_{0})}(f_{P}(x))\subset \Sigma_{1}$ for any point $x\in F_{\gamma}$, $P\in \cp(M_{0})$, here $\ec^{(M_{0})}(f_{P}(x))$ denotes the convex hull of $H_{B'}(f_{P}(x))$ in $\ka_{A}^{M_{0}}$ for $B'\in \cp^{M_{0}}(A)$. On the other hand, $\pi_{M_{0}}(\ec(x))\subset \ec_{M_{0}}(x_{0})$ by definition of $F_{\gamma}$. By the orthogonal decomposition $\ka_{A}^{G}=\ka_{A}^{M_{0}}\oplus \ka_{M_{0}}^{G}$ and the fact that all $\ec(x)$ are positive $(G,T)$-orthogonal family, we see that there exists a bounded convex polytope $\Sigma_{2}$ such that $\ec(x)\subset \Sigma_{2}, \,\forall x\in F_{\gamma}$. By suitably enlarging $\Sigma_{2}$, we can assume that $\Sigma_{2}=\ec((w\lambda)_{w\in W})$ for some dominant cocharacter $\lambda\in X_{*}(A)$ with $\ep^{\lambda}$ lying on the connected component of $\xx$ containing $x_{0}$. By example \ref{tschubert}, we have $F_{\gamma}\subset \sch(\lambda)$, so it must be of finite type.

The assertion that $\xx_{\gamma}=\Lambda\cdot F_{\gamma}$ is implied by the construction of $F_{\gamma}$. The last assertion is due to the fact that any two distinct translations of $F_{\gamma}$ contains no regular points in common.

\end{proof}

\subsection{Examples for $\gl_{d}$}

Let $G=\gl_{d}$, let $T$ be the maximal torus of $G$ of the diagonal matrices, let $B_{0}$ be the Borel subgroup of $G$ of the upper triangular matrices and $B_{0}^{-}$ the opposite of $B_{0}$ with respect to $T$. For each regular element $\gamma\in \kt(\co)$, we have a particular choice of a regular point $x_{0}$ on $\xx_{\gamma}$, which we call \emph{the Kostant regular point}. Let $x_{0}\in \xx_{\gamma}$ be the point representing the lattice $\co[\gamma]$ sitting inside $F[\gamma]\cong F[X]/(X-\gamma_{1})\oplus\cdots \oplus F[X]/(X-\gamma_{d})\cong F^{d}$, where $\gamma_{i}$ are the eigenvalues of $\gamma$. Taking $\{1,\gamma,\cdots, \gamma^{d-1}\}$ as a basis of $F[\gamma]$, we check easily that $x_{0}$ is a regular point.

\begin{prop}\label{classical example}

For $\sigma\in \ks_{d}$, we have 
$$
f_{\sigma B_{0}^{-}}(x_{0})=\sigma\left[\co\oplus \bigoplus_{i=2}^{d}\kp^{\sum_{j=1}^{i-1}\val(\alpha_{\sigma^{-1}(i),\sigma^{-1}(j)}(\gamma))}\right].
$$

\end{prop}

\begin{proof}

Let $\{e_{1},\cdots,e_{d}\}$ be the natural basis of $F^{d}$, the vectors $\gamma^{s}\in \co[\gamma],\, s=0,1,\cdots, d-1,$ correspond to the vectors $\sum_{i=1}^{d}\gamma_{i}^{s}e_{i}$ in $F^{d}$. Let $g$ be the matrix
$$
\begin{bmatrix}
1&\gamma_{1}&\cdots&\gamma_{1}^{d-1}\\
1&\gamma_{2}&\cdots &\gamma_{2}^{d-1}\\
\vdots&\vdots&&\vdots\\
1&\gamma_{d}&\cdots &\gamma_{d}^{d-1}
\end{bmatrix},
$$
then $\co[\gamma]=g \co^{d}$. From this expression and the equality $f_{\sigma B_{0}^{-}}(x_{0})=\sigma\big(f_{B_{0}^{-}}(\sigma^{-1}(x_{0}))\big)$, we see that it suffices to prove the proposition for the standard $B_{0}^{-}$.

After certain elementary operations on the columns, the matrix $g$ can be put in lower triangular form with $1,\,\gamma_{2}-\gamma_{1},\,(\gamma_{3}-\gamma_{2})(\gamma_{3}-\gamma_{1}),\cdots, \prod_{i=1}^{d-1}(\gamma_{d}-\gamma_{i})$ on the diagonal from top to bottom, from which the claim for $f_{B_{0}^{-}}(x_{0})$ follows easily.  
\end{proof}

Let $\gamma\in \kt(\co)$ be regular in minimal form, suppose that its valuation data $(n_{1},\cdots,n_{d-1})$ satisfies $n_{1}\leq n_{2}\leq \cdots \leq n_{d-1}$, then the fundamental domain $F_{\gamma}$ can be written as the intersection of $\xx_{\gamma}$ with two affine Schubert varieties. First of all, we identify $X_{*}(T)$ with $\bz^{d}$ in the natural way. We fix
\begin{eqnarray*}
\mu&=&\left(0,\,n_{1},\,n_{1}+n_{2},\,\cdots,\,\sum_{i=1}^{d-1}n_{i}\right),\\
\lambda&=&\left((d-1)n_{1},\,n_{1}+(d-2)n_{2},\,n_{1}+n_{2}+(d-3)n_{3},\,\cdots,\sum_{i=1}^{d-1}n_{i},\,\sum_{i=1}^{d-1}n_{i}\right).
\end{eqnarray*}
Observe that
\begin{equation}\label{coincidence}
\mu_{i}=\sum_{j=1}^{i-1}\val(\alpha_{j,i}(\gamma)),\quad \lambda_{i}=\sum_{\substack{j=1\\j\neq i}}^{d}\val(\alpha_{j,i}(\gamma)).\end{equation}

\begin{prop}
In the above setting, the fundamental domain $F_{\gamma}$ is the intersection
$$
F_{\gamma}=\xx_{\gamma}\cap\big[\sch(\mu)\cap \ep^{\lambda}\cdot\sch(-\mu)\big].
$$
\end{prop}

\begin{proof}

Let $D_{1}, \,D_{2}$ be the convex polytope with vertices at $(\sigma (\mu))_{\sigma\in \mathfrak{S}_{d}}$ and $(\lambda-\sigma (\mu))_{\sigma\in \mathfrak{S}_{d}}$ respectively. According to the example \ref{tschubert}, we have
$$
\xx^{\vert\mu\vert}(D_{1})=\sch(\mu),\quad \xx^{\vert\mu\vert}(D_{2})=\ep^{\lambda}\cdot\sch(-\mu),
$$
where $\xx^{\vert\mu\vert}$ is the connected component of $\xx$ containing $\ep^{\mu}$. So we only need to prove that $\ec(x_{0})=D_{1}\cap D_{2}$.

Firstly, we check that $\ec(x_{0})\subset D_{1}\cap D_{2}$. For the inclusion in $D_{1}$, observe that both $(f_{B}(x_{0}))_{B\in \cp(T)}$ and $(\sigma(\mu))_{\sigma\in \mathfrak{S}_{d}}$ are positive $(G,T)$-orthogonal families, so it suffices to show that 
$$
f_{\sigma B_{0}^{-}}(x_{0})\prec_{\sigma B_{0}^{-}} \sigma(\mu),
$$
where for $\nu,\nu'\in X_{*}(T)$, $\nu\prec_{\sigma B_{0}^{-}} \nu'$ means that $\nu'-\nu$ is a positive linear combination of the positive coroots with respect to $\sigma B_{0}^{-}$. By proposition \ref{classical example}, this is equivalent to
\begin{eqnarray*}
\sigma\Big(0, \val(\alpha_{\sigma^{-1}(2),\sigma^{-1}(1)}(\gamma)),\cdots, \sum_{j=1}^{d-1}\val(\alpha_{\sigma^{-1}(d),\sigma^{-1}(j)}(\gamma)) \Big)\\
\prec_{\sigma B_{0}^{-}} \sigma\Big(0,\,n_{1},\,n_{1}+n_{2},\,\cdots,\,\sum_{i=1}^{d-1}n_{i}\Big).
\end{eqnarray*}
Permute the inequality by $\sigma^{-1}$, we see that this is a simple consequence of the assumption $n_{1}\leq n_{2}\leq \cdots \leq n_{d-1}$.

The inclusion $\ec(x_{0})\subset D_{1}$ can be used to prove $\ec(x_{0})\subset  D_{2}$ by duality. Define an algebraic involution $\iota$ on $\xx$ by sending $gK$ to $(g^{-1})^{t}K$ for any $g\in G(F)$, where the superscript $^{t}$ means transposition. Since $\gamma$ is diagonal, $\iota$ induces an involution of $\xx_{\gamma}$. In particular, it sends $x_{0}$ to another regular point $x_{0}'$. Looking at how $\iota$ acts on $\xx^{T}$, it is easy to see that $\ec(x_{0}')=-\ec(x_{0})$, i.e. 
$$
f_{\sigma B_{0}}(x_{0}')=-f_{\sigma B_{0}^{-}}(x_{0}), \quad \forall \,\sigma\in \mathfrak{S}_{d},
$$ 
On the other hand, we have $\ec(x_{0})=\ec(x_{0}')+\lambda'$ for some element $\lambda'\in X_{*}(T)$ since both $x_{0}$ and $x_{0}'$ are regular points. It is easy to find that 
\begin{eqnarray*}
\lambda'=f_{B_{0}}(x_{0})-f_{B_{0}}(x_{0}')=f_{B_{0}}(x_{0})+f_{B_{0}^{-}}(x_{0})=\lambda
\end{eqnarray*}
So we get $\ec(x_{0})=\lambda-\ec(x_{0})$. Combined with the inclusion $\ec(x_{0})\subset D_{1}$, we get the other inclusion $\ec(x_{0})\subset D_{2}$.

Given a $(G,T)$-orthogonal family $D$, given $P=MN\in \cf(T)$ maximal, let $d_{P}(D)$ be the distance between the two opposite faces $D^{P}$ and $D^{P^{-}}$. To finish the proof of the proposition, it suffices then to prove that 
$$
d_{P}(\ec(x_{0}))=d_{P}(D_{1}\cap D_{2}).
$$

Choose a minimal gallery of Borel subgroups $B_{1},\cdots,B_{l+1}$ such that $B_{1}\in P,\,B_{l+1}\in P^{-}$, then $\alpha_{B_{i},B_{i+1}},\,i=1,\cdots,l$ runs through $\Phi(N,T)$ exactly once. So we have
\begin{eqnarray*}
d_{P}(\ec(x_{0}))&=& \varpi_{P}(H_{B_{1}}(x_{0})-H_{B_{l+1}}(x_{0}))
=\sum_{i=1}^{l}\varpi_{P}(H_{B_{i}}(x_{0})-H_{B_{i+1}}(x_{0}))\\
&=& \sum_{\alpha\in \Phi(N,T)}\varpi_{P}(\val(\alpha(\gamma))\cdot\alpha^{\vee})=\sum_{\alpha\in \Phi(N,T)}\val(\alpha(\gamma)).
\end{eqnarray*}

Let $P_{i}=M_{i}N_{i}\in \cf(T)$ be the maximal parabolic subgroup associated to the simple root $\alpha_{i}$. Since $d_{P}(D_{1})=d_{P}(D_{2})$ and both $\lambda$ and $\mu$ are anti-dominant, we have
\begin{eqnarray*}
d_{P}(D_{1}\cap D_{2})&=&2\varpi_{P_{i}^{-}}(\mu)-\varpi_{P_{i}^{-}}(\lambda)
=-2\varpi_{i}(\mu)+\varpi_{i}(\lambda)\\
&=& 2\left(\mu_{i+1}+\cdots+\mu_{d}-\frac{d-i}{d}(\mu_{1}+\cdots+\mu_{d})\right)\\
&&-\left(\lambda_{i+1}+\cdots+\lambda_{d}-\frac{d-i}{d}(\lambda_{1}+\cdots+\lambda_{d})\right)\\
&=& \sum_{\alpha\in \Phi(N_{i},T)}\val(\alpha(\gamma)),
\end{eqnarray*}
here we use equation (\ref{coincidence}) in the last equality. Conjugate the above calculation by $\sigma\in \mathfrak{S}_{d}$, we found that 
$$
d_{P}(D_{1}\cap D_{2})=\sum_{\alpha\in \Phi(N,T)}\val(\alpha(\gamma))=d_{P}(\ec(x_{0})).
$$

\end{proof}

\section{Arthur-Kottwitz reduction}

Given a regular semisimple integral element $\gamma\in \kg(F)$, we have the tori $S,T\subset G_{F}$ and the Levi subgroup $M_{0}\in \cl(A)$ as in \S 2.2.
Fix a regular point $x_{0}\in \xx_{\gamma}^{\reg}$. Fix $P_{0}=M_{0}N_{P_{0}}\in \cp(M_{0})$ containing $B_{0}$. Let $\xi\in \ka_{M_0}^{G}$ be such that $\alpha(\xi)$ is positive but almost equal to $0$ for any $\alpha\in \Delta\cap \Phi(N_{P_{0}}, A)$.  Let $D_{0}=(\lambda_{P})_{P\in \cp(M_0)}$ be the $(G,M_0)$-orthogonal family given by
$$
\lambda_{P}=H_{P}(x_{0})+w\cdot \xi,
$$
where $w\in W$ is any element satisfying $P=w\cdot P_{0}$. For $Q=MN\in \cf(M_0)$, define $R_{Q}$ to be the subset of $ \ka_{M_0}^{G}$ satisfying conditions
\begin{eqnarray*}
\pi^{M}(a)&\subset& D_{0}^{Q};\\
\alpha(\pi_{M}(a))&\geq& \alpha(\pi_{M}(\lambda_{P})),\,\forall\,\alpha\in \Phi(N,A),\,\forall\, P\in \cp(M_0),\, P\subset Q.
\end{eqnarray*}
Notice that $R_{G}=D_{0}$. We get a partition 
\begin{equation}\label{part}
\ka_{M_0}^{G}=\bigcup_{Q\in \cf(M_0)} R_{Q}.
\end{equation} 
The figure \ref{partitiongl3} gives an illustration of the partition  for the group $\gl_{3}$ and $M_{0}=T=A$. The partition (\ref{part}) induces a disjoint partition of $\Lambda_{M_{0}}$ via the map $\Lambda_{M_0}\to \ka_{M_0}^{G}$, since we have perturbed the $(G,M_0)$-family $(H_{P}(x_{0}))_{P\in \cp(M_0)}$ with $\xi$. We also want to point out that the partition (\ref{part}) has good  transitivity property: For a maximal parabolic $Q=LU\in \cf(M_0)$, the partition $\overline{R}_{Q}:=\bigcup_{Q'\subset Q}R_{Q'}$ gives similar partition of $\ka_{M_{0}}^{L}$ as (\ref{part}) under the natural projection $\ka_{M_{0}}^{G}\to \ka_{M_{0}}^{L}$.

\begin{figure}[h]
\begin{center}
\begin{tikzpicture}[scale=0.6]

\draw (-2.5,3.17)--(0.5,3.17);
\draw (-4,0.57)--(-1.5,-3.75);
\draw (1.5,-3.75)--(3,-1.15);
\draw (-4,0.57)--(-2.5,3.17);
\draw (-1.5,-3.75)--(1.5,-3.75);
\draw (0.5,3.17)--(3,-1.15);

\draw [red] (-2.5,3.17)--(-2.5, 5.17);
\draw [red] (0.5,3.17)--(0.5, 5.17);
\draw [red] (-1.5,-3.75)--(-1.5, -5.75);
\draw [red] (1.5,-3.75)--(1.5, -5.75);
\draw [red] (-4.23,4.17)--(-2.5,3.17);
\draw [red] (2.23,4.17)--(0.5,3.17);
\draw [red] (-4,0.57)--(-5.75, 1.57);
\draw [red] (-4,0.57)--(-5.75, -0.43);
\draw [red] (-1.5,-3.75)--(-3.23,-4.75);
\draw [red] (1.5,-3.75)--(3.23,-4.75);
\draw [red] (4.73,0.15)--(3,-1.15);
\draw [red] (4.73,-2.15)--(3,-1.15);

\node [blue] at (-2.5, -5.5) {$R_{B^{-}}$};
\node [blue] at (-4.5, -2.6) {$R_{P^{-}}$};
\node [blue] at (3,2) {$R_{P}$};
\node [blue] at (1.4,4.6) {$R_{B}$};
\node [blue] at (0,0) {$D_{0}$};

\node at (-4.2, 3.8) {$\bullet$};
\node [black] at (-5.5,3.8) {$H_{P}(x)$};
\node [black] at (-1,5.2) {$\overline{R}_{Q_{0}}$};
\node [black] at (-6,3) {$\overline{R}_{Q_{1}}$};

\draw (-4.2,3.8)--(-0.5,3.8);
\draw [blue] (-1,3.8)--(-1,3.17);
\draw [blue] (-4.2,3.8)--(-2.66,2.9);

\end{tikzpicture}
\caption{Partition of $\ka_{A}^{G}$ for $\gl_{3}$.}
\label{partitiongl3}
\end{center}
\end{figure}
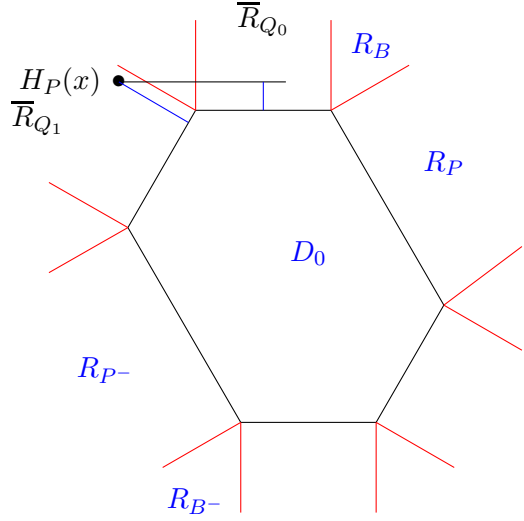

\begin{lem}
For any $x\in \xx_{\gamma}$, there exists a unique $Q\in \cf(M_0)$ such that $\ec_{M_{0}}^{Q}(x)\subset R_{Q}$. 
\end{lem}

\begin{proof}

The uniqueness is clear for the regular points $x\in \xx_{\gamma}^{\reg}$, since $\ec_{M_{0}}(x)$ is a translation of $\ec_{M_{0}}(x_{0})$ by $\Lambda_{M_{0}}$. By proposition \ref{gkmreg}, for any $x\in \xx_{\gamma}$, the convex polytope $\ec_{M_{0}}(x)$ is contained in a translation of $\ec(x_{0})$ by some $\lambda\in \Lambda$, from which the uniqueness for general case follows. 

Now we prove the existence. We can suppose that $x\notin \xx_{\gamma}(D_{0})$. 
For maximal parabolic $Q\in \cf(M_0)$, notice that $\varpi_{Q}(H_{B}(x))$ doesn't depend on the choice of $B\in \cp(A),\,B\subset Q$, we write it as $N_{Q}(x)$. Let $Q_{0}$ be the maximal parabolic such that 
$$
N_{Q_{0}}(x)-N_{Q_{0}}(x_{0})=\max_{Q\in \cf(M_{0})\text{ maximal}}\{N_{Q}(x)-N_{Q_{0}}(x_{0})\},
$$
then $N_{Q_{0}}(x)-N_{Q_{0}}(x_{0})>0$ since $x\notin \xx_{\gamma}(D_{0})$.

We claim that $\ec_{M_{0}}^{Q_{0}}(x)\subset \overline{R}_{Q_{0}}$. If this is not the case, there exists $P\in \cp(M_{0}),\,P\subset Q_{0}$, such that $H_{P}(x)\notin \overline{R}_{Q_{0}}$. Since $N_{Q_{0}}(x)-N_{Q_{0}}(x_{0})>0$, there exists a maximal parabolic subgroup $Q_{1}\in \cf(M_{0})$ which is adjacent to $Q_{0}$, such that $H_{P}(x)\in \overline{R}_{Q_{1}}$. The situation is best illustrated by the upper left corner of figure \ref{partitiongl3}. Since in a right triangle with sides $a,b,c$, we always have $c>a,b$, we get
$$
N_{Q_{1}}(x)-N_{Q_{1}}(x_{0})> N_{Q_{0}}(x)-N_{Q_{0}}(x_{0}),
$$
which is a contradiction to the assumption on $Q_{0}$.

Now we can use the retraction $f_{Q_{0}}:\xx_{\gamma}\to \xx_{\gamma}^{M_{Q_{0}}}$ to find the required parabolic subgroup $Q$ inductively.

\end{proof}

With this lemma, we define 
$
S_{Q}:=\{x\in \xx_{\gamma}\mid \ec_{M_{0}}^{Q}(x)\subset R_{Q}\}.
$
Notice that $F_{\gamma}$ is one connected component of $S_{G}=\xx_{\gamma}(D_{0})$. As in example \ref{diffcomp}, the other connected components of $S_{G}$ may have slight difference from $F_{\gamma}$. We get a disjoint partition 
$$
\xx_{\gamma}=\xx_{\gamma}(D_{0})\cup\bigcup_{\substack{Q\in \cf(M_0)\\Q\neq G}}S_{Q}.
$$  

For each parabolic subgroup $Q=MN\in \cf(M_{0})$, consider the restriction of the retraction $f_{Q}: \xx\to \xx^{M}$ to $S_{Q}$, its image is $S_{Q}\cap \xx^{M}$. Recall that the connected components of $\xx^{M}$ are fibers of the map $\nu_{M}:\xx^{M}\to \Lambda_{M}$. For $\nu\in \Lambda_{M}$, let $\xx^{M,\nu}$ be its fiber at $\nu$. Let $S_{Q}^{\nu}=f_{Q}^{-1}(S_{Q}\cap \xx^{M,\nu})$, it is easy to verify that
$$
S_{Q}\cap \xx^{M,\nu}=\xx_{\gamma}^{M,\nu}(D_{0}^{Q}).
$$

\begin{prop}\label{klretraction}
The retraction $f_{Q}: S_{Q}^{\nu}\to \xx_{\gamma}^{M,\nu}(D_{0}^{Q})$ is an iterated fibration in affine spaces.
\end{prop}

\begin{proof}

Since $f_{Q}(ux)=f_{Q}(x),\,\forall\,u\in N,\,x\in \xx$, by the definition of $S_{Q}^{\nu}$, we have the identity
$$
S_{Q}^{\nu}=[N(F)\cdot \xx_{\gamma}^{M,\nu}(D_{0}^{Q})]\cap \xx_{\gamma}.
$$
So the fiber of $f_{Q}:S_{Q}^{\nu}\to \xx_{\gamma}^{M,\nu}(D_{0}^{Q})$ at $mM(\co)$ is 
$$
\Big\{umM(\co)\mid u\in N(F),\,\Ad(u^{-1})\gamma\in \Ad(m)\kg(\co)\Big\}.
$$ 
We'll prove that they form a family which is an iterated fibration in affine spaces.

We follow the proof of Kazhdan-Lusztig \cite{kl}, \S5. By assumption, $\mathrm{char}(k)>\rk(G)$, the exponential map $\exp:\kn\to N$ is well defined. The group $N$ has the decreasing filtration by normal subgroups
$$
N_{0}=N\supset N_{1}=[N,N]\supset \cdots\supset N_{i}=[N_{i-1},N]\supset\cdots \supset N_{\rk(G)}\supset 1.
$$ 
The exponential map induces an isomorphism $\kn_{i}/\kn_{i+1}\to N_{i}/N_{i+1}$ which sends $n_{i}$ to $1+n_{i}$.

Let $\ck$ be the $K$-equivariant fiber bundle $G(F)\times_{K} K$ on $\xx$, let $\kkk$ be the $K$-equivariant vector bundle $G(F)\times_{K}\kg(\co)$ on $\xx$, where $K$ acts on $K$ and $\kg(\co)$ by conjugaison. Let $\wnn_{i}$ be the constant fiber bundle $\xx\times N_{i}(F)$, let $\wn_{i}$ be the constant vector bundle $\xx\times \kn_{i}(F)$. We denote also $\wnn=\wnn_{0}$.

To begin with, observe that with the retraction $f_{Q}$, the locally closed subvariety $f_{Q}^{-1}(\xx^{M,\nu})$ of $\xx$ is isomorphic to the restriction of the fiber bundle $\wnn/\wnn\cap \ck$ over $ \xx^{M,\nu}$, we will identify them in the following. For $i=0,\cdots,\rk(G)+1$, let $S_{i}$ be the sub bundle of $\wnn_{i}\backslash \wnn/\wnn\cap \ck$ restricted to $\xx_{\gamma}^{M,\nu}(D_{0}^{Q})$, whose fiber at $mM(\co)$ is given by
$$
\Big\{u\in N_{i}(F)\backslash N(F)/N(F)\cap \Ad(m)K\mid \Ad(u)^{-1}\gamma\in \Ad(m)\kg(\co)+\kn_{i}(F) \Big\}.
$$
Let $p_{i}:S_{i+1}\to S_{i}$ be the natural projection, we get a tower of projections 
$$
S_{Q}^{\nu}\cong S_{\rk(G)+1}\to S_{\rk(G)}\to\cdots \to S_{0}\cong \xx_{\gamma}^{M,\nu}(D_{0}^{Q}).
$$
The last isomorphism is due to the equivalence of the equations $\gamma \in \Ad(m)\kg(\co)+\kn(F)$ and $\gamma \in \Ad(m)\kg(\co)$ since $\Ad(m)^{-1}\gamma\in \km(F)$. We will prove that each $S_{i+1}$ is a homogeneous space under a vector bundle over $S_{i}$, this will end the proof of the proposition.

Given $gK\in S_{i}$, we have
$$
\gamma\in \Ad(g)\kg(\co)+\kn_{i}(F).
$$
Let $u=1+n\in N_{i+1}(F)\backslash N_{i}(F)$, with $n\in \kn_{i+1}(F)\backslash \kn_{i}(F)$, then
\begin{eqnarray}
ugK\in S_{i+1}&\iff& \Ad(u^{-1})\gamma \in \Ad(g)\kg(\co)+\kn_{i+1}(F)\nonumber \\
\label{klequ}&\iff& \gamma+[\gamma,n]\in \Ad(g)\kg(\co)+\kn_{i+1}(F).
\end{eqnarray}
Using the isomorphism
$$
\frac{\Ad(g)\kg(\co)+\kn_{i}(F)}{\Ad(g)\kg(\co)+\kn_{i+1}(F)}\cong \frac{\kn_{i}(F)/\kn_{i}(F)\cap \Ad(g)\kg(\co)}{\kn_{i+1}(F)/\kn_{i+1}(F)\cap \Ad(g)\kg(\co)},
$$
let $\bar{\gamma}$ be the image of $\gamma$ under the isomorphism, then the equation (\ref{klequ}) means that $n$ should satisfy the equation $\ad(\gamma)n=-\bar{\gamma}$ in the above quotient. Consider the endomorphism $\ad(\gamma)$ of the restriction of the  vector bundle
\begin{equation}\label{torseur}
\frac{\wn_{i}/\wn_{i}\cap \kkk}{\wn_{i+1}/\wn_{i+1}\cap \kkk}
\end{equation}
on $S_{i}$. It is surjective since $\ad(\gamma):\kn_{i}(F)\to \kn_{i}(F)$ is. This means that there is always $n$ such that equation (\ref{klequ}) is satisfied, i.e. $p_{i}$ are surjective for all $i$. Further more, let $V_{i}$ be kernel of the endomorphism $\ad(\gamma)$ of the vector bundle (\ref{torseur}), then $V_{i}$ is a vector bundle on $S_{i}$. The above calculation shows that $S_{i+1}$ is a homogeneous space over $S_{i}$ under the vector bundle $V_{i}$.

\end{proof}

\begin{prop}\label{locally closed}
The strata $S_{Q}^{\nu}$ are locally closed sub varieties of $\xx_{\gamma}$. Furthermore, in the decomposition
\begin{equation}\label{arthurkottwitz}
\xx_{\gamma}=\xx_{\gamma}(D_{0})\cup \bigcup_{\substack{Q\in \cf(M_0)\\ Q\neq G}}\bigcup_{\nu\in \Lambda_{M_{Q}}\cap R_{Q}} S_{Q}^{\nu},
\end{equation}
we can order the strata $S_{Q}^{\nu}$ as $S^{(1)},S^{(2)}, \cdots,$ such that for each $n\in \bn$, the union $\bigcup_{i=1}^{n}S^{(i)}$ is a closed sub variety of $\xx_{\gamma}$.
\end{prop}

\begin{proof}

To begin with, $\xx_{\gamma}(D)$ is a closed sub variety of $\xx_{\gamma}$ for any $(G,M_0)$-orthogonal family $D$. Now we prove by induction. Let $Q_{0}=L_{0}U_{0}\in \cf(M_0)$ be a maximal parabolic subgroup containing $Q$. For $Q'=M'N'\in \cf(M_0),\,Q'\subset Q_{0}$, let $p^{M'}_{L_{0}}$ be the natural projection  $\Lambda_{M'}\to \Lambda_{L_{0}}$, let $\nu_{0}=p^{M}_{L_{0}}(\nu)$. Consider 
$$
Z_{Q_{0}}^{\nu_{0}}:=\bigcup_{Q'\subset Q_{0}}\bigcup_{p^{M'}_{L_{0}}(\nu')=\nu_{0}}S_{Q'}^{\nu'}.
$$
Firstly, $Z_{Q_{0}}^{\nu_{0}}$ can be written as a difference $\xx_{\gamma}(D)\backslash \xx_{\gamma}(D')$ for two $(G,M_0)$-orthogonal family $D,\,D'$. Secondly, observe that 
$$
S_{Q'}^{\nu'}=[U_{0}(F)\cdot (S_{Q'}^{\nu'}\cap \xx^{L_{0},\nu_{0}})]\cap \xx_{\gamma},
$$ 
the same proof as that of proposition \ref{klretraction} shows that the retraction 
$$
f_{Q_{0}}:Z_{Q_{0}}^{\nu_{0}}\to \xx_{\gamma}^{L_{0},\nu_{0}}(D^{Q_{0}})
$$
is an iterated fibration in affine spaces. Now the claim follows by induction, using the transitivity property of $f_{P}$.
\end{proof}

By proposition \ref{klretraction}, each strata $S_{Q}^{\nu}$ has an iterated affine fibration onto $\xx_{\gamma}^{M_{Q},\nu}(D_{0}^{Q})$, so the study of $\xx_{\gamma}$ is reduced to that of $F_{\gamma}$. We call the decomposition (\ref{arthurkottwitz}) \emph{the Arthur-Kottwitz reduction}.




\begin{lem}\label{fundamental variant}
Suppose that $F_{\gamma}^{M}$ is cohomologically pure for any proper Levi subgroup $M$ of $G$ containing $M_0$. Suppose that $F_{\gamma}$ is cohomologically pure, then the truncated affine Springer fiber $\xx_{\gamma}^{\nu}(D_{0})$ is cohomologically pure for all $\nu\in \Lambda_{G}$.

 \end{lem}

\begin{proof}

As in the example \ref{diffcomp}, after certain translation on $\xx$ by $\Lambda_{M_{0}}$, we have $F_{\gamma}=\xx_{\gamma}^{\nu}(D_{0}+\varpi^{\vee})$, where $\varpi^{\vee}$ is the image of a minuscule coweight in $\Lambda_{M_{0}}$, and $D_{0}+\varpi^{\vee}$ is the translation of $D_{0}$ by $\varpi^{\vee}$. It is easy to see that $\xx_{\gamma}^{\nu}(D_{0})\subset F_{\gamma}$. Applying the reduction of Arthur-Kottwitz, the open sub variety $F_{\gamma}\backslash \xx_{\gamma}^{\nu}(D_{0})$ is naturally stratified into finite unions of $S_{Q}^{\nu'}\cap F_{\gamma}$. Since the two truncation parameters differ by a minuscule coweight, by proposition \ref{gkmreg}, we have $S_{Q}^{\nu'}\subset F_{\gamma}$, so $S_{Q}^{\nu'}\cap F_{\gamma}=S_{Q}^{\nu'}$. 

Now we proceed by induction. Suppose that the lemma is proved for all the Levi subgroups $M\in \cl(M_0)$, then $\xx_{\gamma}^{M,\nu'}(D_{0}^{Q})$ are all cohomologically pure for all $Q\in \cp(M)$ and all $\nu'\in \Lambda_{M}$. By proposition \ref{klretraction} and \ref{locally closed}, we see that $F_{\gamma}\backslash \xx_{\gamma}^{\nu}(D_{0})$ is cohomologically pure. Now the long exact sequence
$$\cdots\to H^{i-1}(\xx_{\gamma}^{\nu}(D_{0}))\to H^{i}_{c}(F_{\gamma}\backslash \xx_{\gamma}^{\nu}(D_{0}))\to H^{i}(F_{\gamma})\to H^{i}(\xx_{\gamma}^{\nu}(D_{0}))\to \cdots 
$$ 
will split into short exact sequence
$$
0 \to H^{i}_{c}(F_{\gamma}\backslash \xx_{\gamma}^{\nu}(D_{0}))\to H^{i}(F_{\gamma})\to H^{i}(\xx_{\gamma}^{\nu}(D_{0}))\to 0,
$$
because $H^{i-1}(\xx_{\gamma}^{\nu}(D_{0}))$ is of weight at most $i-1$ by \cite{weil2}. The claim then follows from the above short exact sequence.

\end{proof}

Now we come to the proof of theorem \ref{main1}. We will prove a slightly stronger result. A positive $(G,M_0)$-orthogonal family $D=(\mu_{P})_{P\in \cp(M_0)}$ is said to be \emph{regular} with respect to $D_{0}$ if $\mu_{P}\in R_{P},\,\forall\,P\in \cp(M_0)$.

\begin{thm}\label{equivalence of two conjectures}
Suppose that $F_{\gamma}^{M}$ is cohomologically pure for any proper Levi subgroup $M$ of $G$ containing $M_0$. Let $D$ be a positive $(G,M_0)$-orthogonal family which is regular with respect to $D_{0}$. Then the truncated affine Springer fiber $\xx_{\gamma}(D)$ is cohomologically pure if and only if $F_{\gamma}$ is. 
\end{thm}

\begin{proof}

The complication that some connected components of $\xx_{\gamma}(D)$ don't contain $F_{\gamma}$ is already treated in lemma \ref{fundamental variant}, so we can suppose that every connected component of $\xx_{\gamma}(D)$ contains a translation of $F_{\gamma}$. Applying the Arthur-Kottwitz reduction to every connected component $\xx_{\gamma}^{\nu}(D)$, we get a stratification of $\xx_{\gamma}^{\nu}(D)\backslash F_{\gamma}$ into finite union of $S_{Q}^{\nu'}\cap \xx_{\gamma}^{\nu}(D)$. The hypothesis that $D$ is regular with respect to $D_{0}$ implies that each $S_{Q}^{\nu'}$ is either contained in $\xx_{\gamma}^{\nu}(D)$ or disjoint from it. Applying lemma \ref{fundamental variant} to the Levi subgroups $M\in \cl(M_0)$, we see that all the truncated affine Springer fibers $\xx_{\gamma}^{M,\nu'}(D_{0}^{Q}),\,Q\in \cp(M)$ are cohomologically pure, which implies that $\xx_{\gamma}^{\nu}(D)\backslash F_{\gamma}$ is cohomologically pure by proposition \ref{klretraction} and \ref{locally closed}. Now the theorem follows from the same argument as the last part of the proof of lemma \ref{fundamental variant}.

\end{proof}

\begin{rem}\label{clpurity}

When $T$ splits over $F$, Chaudouard and Laumon have conjectured that $\xx_{\gamma}(D)$ is cohomologically pure whenever $D$ is \emph{sufficiently regular}, see \cite{cl} conjecture 1.3. Here $D=(\lambda_{B})_{B\in \cp(T)}$ is said to be \emph{sufficiently regular} if $\lambda_{B}$ lies in the chamber indexed by $B$ and is \emph{sufficiently} far from the walls bounding the chamber. However, they don't have a lower bound on the distance of $\lambda_{B}$ to the walls for $D$ to be sufficiently regular, and the word \emph{sufficiently} should be understood as \emph{as large as necessary}. In this sense, when $D$ is sufficiently regular, it is regular with respect to $D_{0}$. By our theorem \ref{equivalence of two conjectures}, their purity conjecture is equivalent to the conjectures \ref{gkmconj1} and \ref{gkmconj2}.

\end{rem}

\section{Reformulation of the rationality conjecture}

In this section, we will work with the group $G=\pgl_{d+1}$ instead of $\gl_{d+1}$ to simplify some technical points. We are only interested in the case when $\gamma$ is unramified, i.e. we assume that $T$ is the maximal torus of $G$ of the diagonal matrices. Let $B_{0}$ be the Borel subgroup of $G$ of the upper triangular matrices. We conserve the notations of the previous sections. 

We will assume that $F_{\gamma}$ is cohomologically pure, and give a geometric reformulation of the rationality conjecture \ref{rationality}. The main ingredient of the proof is the calculation by Chaudouard and Laumon \cite{cl} of the $T$-equivariant homology of cohomologically pure truncated affine Springer fibers.

\subsection{Symmetric algebras}

Let 
$$
\bs=\sym(\kt)=\bigoplus_{i=0}^{\infty}\sym^{i}((\kt^{*})^{*})
$$ 
be the ring of polynomial functions on $\kt^{*}$ with coefficients in $\qlbar$. Let
$$
\dd=\sym(\kt^{*})=\bigoplus_{i=0}^{\infty}\sym^{i}(\kt^{*})
$$
be the ring of linear differential operators with constant $\qlbar$-coefficients on $\kt^{*}$. The non-degenerate perfect pairing 
$$
\langle\,,\rangle:\;\dd\times \bs\to \qlbar
$$
given by $\langle \partial, P\rangle=\partial(P)(0)$ satisfies

\begin{equation}\label{sd}
\langle \partial\partial', P\rangle=\langle \partial, \partial' P\rangle.
\end{equation}

It induces a natural duality between the homogeneous degree $n$ pieces $\dd^{n}$ and $\bs_{n}$. For a homogeneous ideal $I=\bigoplus_{i=1}^{\infty}I_{i}\subset \dd$, let 
$$
\bs\{I\}=\{f\in S(\kt)\mid \partial f=0,\,\forall \,\partial\in I\}.
$$

By (\ref{sd}), we have
$$
\bs\{I\}=I^{\perp}=\bigoplus_{i=0}^{\infty}I_{i}^{\perp}.
$$

For each root $\alpha\in \Phi^{+}$, we will denote by $\partial_{\alpha}\in \dd$ the corresponding differential operator on $\kt^{*}$.

\subsection{$T$-equivariant homology}

Let $X$ be a separated $k$-scheme of finite type endowed with an algebraic action of the torus $T$. The $T$-equivariant cohomology of $X$ is defined to be the cohomology of the quotient stack $[X/T]$, i.e.
$$
H^{*}_{T}(X)=H^{*}_{T}(X,\qlbar)=H^{*}([X/T],\qlbar).
$$
It is a $\bz_{\geq 0}$-graded $\qlbar$-algebra with respect to the cup product. Via the structural morphism 
\begin{equation}\label{str}
[X/T]\to BT=[\spec(k)/T],
\end{equation} 
it becomes a graded algebra over the ring $H^{*}_{T}(\spec(k))$, which is isomorphic to the $k$-algebra $\dd$ via the Chern-Weil isomorphism. More precisely, we have natural isomorphism $\dd^{1}=X^{*}(T)\otimes \qlbar$, and given $\chi\in X^{*}(T)$, let $c_{1}(\chi)$ be the first Chern class associated to the resulting line bundle on $BT$. Then $c_{1}$ extends to a degree-doubling isomorphism
$$
\dd\to H^{*}(BT)=H^{*}_{T}(\spec(k)).
$$

The Leray spectral sequence associated to the structural morphism (\ref{str}) is
$$
E_{2}^{p,q}=H^{p}_{T}(\spec(k))\otimes H^{q}(X)\Rightarrow H^{p+q}_{T}(X).
$$
When $X$ is cohomologically pure, the spectral sequence will degenerate at $E_{2}$, and we get a non-canonical isomorphism
\begin{equation}\label{equiord}
H^{i}_{T}(X)\cong  \bigoplus_{p+q=i}H^{p}(X)\otimes \dd^{q}.
\end{equation}

In \cite{gkm1} and \cite{cl}, the authors work with homology instead of cohomology in order to facilitate the process of taking limits. This is defined by taking dualities, for example,
$$
H^{T}_{*}(X)=\Hom\big(H^{*}_{T}(X),\qlbar \big).
$$

Using the natural duality between $\bs$ and $\dd$, the isomorphism (\ref{equiord}) can be rewritten as
$$
H^{T}_{i}(X)\cong \bigoplus_{p+q=i} H_{p}(X)\otimes \bs_{q}.
$$

\subsection{$T$-equivariant homology of $F_{\gamma}$}

Given a positive $(G,T)$-orthogonal family $D$, suppose that $\xx_{\gamma}(D)$ is \emph{cohomologically pure}, Chaudouard and Laumon \cite{cl} have calculated the $T$-equivariant homology of $\xx_{\gamma}(D)$. The result is expressed in terms of the $T$-fixed points and $1$-dimensional $T$-orbits in $\xx_{\gamma}(D)$. We adapt their result to our situation.

For $\bnn\in\bn^{d}$, let $\gamma\in \kt(\co)$ be an element in minimal form with root valuation $\bnn$. Given a regular point $x_{0}\in \xx_{\gamma}^{\reg}$, up to a suitable translation by $\Lambda$, we will assume that $H_{B_{0}^{-}}(x_{0})=0$. Let $E_{\bnn}=\ec(x_{0})$ and let 
$$
\Lambda_{\bnn}:=\{\lambda\in X_{*}(T)\mid \lambda \in \ec(x_{0}),\,\nu_{G}(\ep^{\lambda})=\nu_{G}(x_{0})\}.
$$

For each $\alpha\in \Phi^{+}$, let 
$$
R_{\alpha, i}=\sum_{\lambda \text{ satisfying }(*)} (1-\alpha^{\vee})^{i}\lambda\otimes \bs\{\partial_{\alpha}^{i}\}\subset \qlbar^{\Lambda_{\bnn}} \otimes \bs,
$$
where $(*)$ refers to the condition:
\begin{equation*}
\lambda, \alpha^{\vee}\lambda,\cdots, (\alpha^{\vee})^{i}\lambda \in \Lambda_{\bnn}.
\end{equation*}

\begin{thm}[Chaudouard-Laumon \cite{cl}, prop. 10.3]

Assume that $F_{\gamma}$ is cohomologically pure, then we have the exact sequence
\begin{equation}\label{fundamental}
0\to \sum_{\alpha\in \Phi^{+}}\sum_{i=1}^{\val(\alpha(\gamma))} R_{\alpha, i} \to \qlbar^{\Lambda_{\bnn}}\otimes \bs\to H^{T}_{*}(F_{\gamma})\to 0.
\end{equation}

\end{thm}

We will write the first term in the exact sequence  (\ref{fundamental}) as $R_{\bnn}$. For a $\bz_{\geq 0}$-graded $\qlbar$-vector space $M=\bigoplus_{n=0}^{+\infty}M_{n}$, we will define its Poincar\'e series to be
$$
\sum_{n=0}^{+\infty} \dim(M_{n})q^{n}.
$$
The $\qlbar$-algebra $\bs$ is naturally $\bz_{\geq 0}$-graded, it induces a grading on the three terms in the exact sequence (\ref{fundamental}). We will denote their Poincar\'e series by $Q_{1,\bnn}(q),\,Q_{2,\bnn}(q),\,Q_{3,\bnn}(q)$ respectively. Let 
$$
Q_{i}(q;\vec{t}\;)=\sum_{n_{1}=1}^{+\infty}\cdots\sum_{n_{d}=1}^{+\infty} Q_{i,(n_{1},\cdots,n_{d})}(q)t_{1}^{n_{1}}\cdots t_{d}^{n_{d}},\quad i=1,2,3.
$$

Since $F_{\gamma}$ is assumed to be cohomologically pure, we have isomorphism
$$
H^{T}_{*}(F_{\gamma})\cong  H_{*}(F_{\gamma})\otimes \bs_{*},
$$
from which we deduce that 
$$
Q(q;\vec{t}\;)=(1-q)^{d}Q_{3}(q;\vec{t}\;).
$$
So the rationality of $Q(q;\vec{t}\;)$ is the same as that of $Q_{3}(q;\vec{t}\;)$. By the exact sequence (\ref{fundamental}), this is the same as the rationality of $Q_{2}(q;\vec{t}\;)-Q_{1}(q;\vec{t}\;)$.

\subsection{Toric varieties}

Let $\widehat{T}$ be the dual torus of $T$, we have
$$
X_{*}(\widehat{T})=X^{*}(T),\quad X^{*}(\widehat{T})=X_{*}(T).
$$

Let $\Sigma$ be a complete fan in $\kt^{*}=X_{*}(\wht)\otimes \br$. For $n=0,\cdots, \dim(T)$, let $\Sigma(n)$ be the set of $n$-dimensional cones in $\Sigma$. Let $Y=Y_{\Sigma}$ be the toric compactification of $\wht$ according to the fans $\Sigma$.
To each cone $\sigma\in \Sigma$, we associate a $\wht$-invariant affine open sub variety 
$$
U_{\sigma}:=\spec(\qlbar[\,\sigma^{\vee}\cap X^{*}(\wht)]),
$$ 
where $\sigma^{\vee}=\{\lambda\in X^{*}(\wht)\mid \langle \lambda, x\rangle \geq 0,\,\forall \, x\in \sigma\}$. Putting together, they give an affine covering of $Y$.
The map $\sigma\to U_{\sigma}$ is inclusion preserving. 
We have also an inclusion-reversing bijection $\sigma \to D_{\sigma}$ between the cones and the $\wht$-invariant closed irreducible sub varieties of $Y$. More precisely, $D_{\sigma}$ is contained in the union of the affine open sub varieties $U_{\tau},\,\sigma\subset\tau$. In each $U_{\tau}$, $D_{\sigma}\cap U_{\tau}$ is defined by the ideal generated by
$$
\{\lambda\in X^{*}(\wht)\mid \lambda\in \tau^{\vee},\,\lambda\notin \sigma^{\perp}\},
$$
where $\sigma^{\perp}=\{\lambda\in X^{*}(\wht)\mid \langle \lambda, x\rangle = 0,\,\forall \, x\in \sigma\}$. It is easy to see that
$$
D_{\sigma}\cap U_{\tau}=\spec(\qlbar[\,\tau^{\vee}\cap \sigma^{\perp}\cap X^{*}(\wht)]),
$$ 
and the codimension of $D_{\sigma}$ in $Y$ is equal to the dimension of $\sigma$. In particular, the $D_{\sigma}$'s, for $ \sigma\in \Sigma(1)$, generate the group $\Div_{\wht}(Y)$ of the $\wht$-invariant Weil divisors in $Y$.

For $\sigma\in \Sigma(1)$, let $\varpi_{\sigma}$ be the generator of the semi-group $\sigma\cap X_{*}(\wht)$. To each $\lambda\in X^{*}(\wht)$, viewed as a meromorphic function on $Y$, is associated its principal divisor
$$
(\lambda)=\sum_{\sigma\in \Sigma(1)} \varpi_{\sigma}(\lambda)D_{\sigma}.
$$

Let $\Cl(Y)$ be the class group of Weil divisors on $Y$, then we have the exact sequence
\begin{equation}\label{cly}
0\to X^{*}(\wht)\to \Div_{\wht}(Y)\to \Cl(Y)\to 0.
\end{equation}

For $D\in \Div_{\wht}(Y)$, let $[D]$ be its equivalent class in $\Cl(Y)$. It is said to be \emph{effective} if $D=\sum_{\sigma} n_{\sigma} D_{\sigma} $ with positive coefficients. In this case, we write $D\geq 0$. We write also $D_{1}\geq D_{2}$ if $D_{1}-D_{2}\geq 0$.

The toric variety $Y$ has a quotient construction similar to that of $\bp^{n}$. We introduce the \emph{homogeneous coordinate ring} 
$$
A=\qlbar[\,y_{\sigma};\,\sigma\in \Sigma(1)],
$$
which is graded by the abelian group $\Cl(Y)$ in the following way: To every monomial $\prod_{\sigma}y_{\sigma}^{n_{\sigma}},\,n_{\sigma}\in \bz$, we associate the divisor $D=\sum_{\sigma}n_{\sigma}D_{\sigma}$. This monomial, written $y^{D}$, is of degree $[D]\in \Cl(Y)$. In this way, we get the grading
$$
A=\bigoplus_{[D]\in \Cl(Y)} A[D].
$$
The group $\bbg_{m}^{\Sigma(1)}$ acts naturally on $A$, hence on $\spec(A)$. Let 
$$
\widehat{\Cl}(Y)=\Hom_{\bz}(\Cl(Y),\,\qlbar^{\times}).
$$ 
It is a sub-torus of $\bbg_{m}^{\Sigma(1)}$ if one takes the duality of the exact sequence (\ref{cly}). In this way, it acts on $\spec(A)$ as well. To define the quotient, we need to introduce the irrelevant ideal $B\subset A$, which is generated by the $y^{\hat{\tau}},\,\tau\in \Sigma$, where 
$$
y^{\hat{\tau}}=\prod_{\sigma\in \Sigma(1),\,\sigma\nsubseteq \tau}y_{\sigma}.
$$
Now we have
$$
Y=\big[\spec(A)\backslash \mathbf{V}(B)\big]\sslash\widehat{\Cl}(Y).
$$
Similarly to the case of $\bp^{n}$, we have an exact functor $L\to \widetilde{L}$ from the category of graded $A$-modules to the category of quasi-coherent sheaves on $Y$. It sends $A$ to the structure sheaf and finitely generated graded $A$-modules to coherent sheaves on $A$. Furthermore, all the quasi-coherent sheaves on $A$ are of the form $\widetilde{L}$ for some graded $A$-modules $L$.

Let $D=\sum_{\sigma}n_{\sigma} D_{\sigma}$ be a $\wht$-invariant Weil divisor on $Y$. We associate to it a lattice polytope $P_{D}$ in $\kt=X^{*}(\wht)\otimes \br$:
$$
P_{D}=\left\{a\in \kt| \langle \varpi_{\sigma},\,a\rangle + n_{\sigma}\geq 0,\;\forall\,\sigma\in \Sigma(1)  \right\}.
$$
The defining inequality is reminiscent of the inequality $(\lambda)+D\geq 0$, and we have

\begin{prop}

Let $D$ be a $\wht$-invariant Weil divisor on $Y$, then we have
$$
\Gamma(Y, \co_{Y}(D))=\bigoplus_{\lambda\in X^{*}(\wht)\cap P_{D}} \bc\cdot \lambda.
$$

\end{prop}

This is related to the homogeneous coordinate ring $A$ as follows: Let 
$$
A'=\qlbar[y_{\sigma}^{\pm 1};\,\sigma\in \Sigma(1)].
$$ 
The pull back of rational functions corresponding to the natural projection 
$$
\spec(A)\backslash \mathbf{V}(B)\to Y
$$
induces an injective morphism $\qlbar[X^{*}(\wht)]\to A'$ given by
\begin{equation}\label{lambda function}
\lambda\in X^{*}(\wht)\longmapsto y^{\lambda}=\prod_{\sigma\in \Sigma(1)}y_{\sigma}^{\varpi_{\sigma}(\lambda)}.
\end{equation}
Given $\lambda\in X^{*}(\wht)\cap P_{D}$, we define its $D$-homogenization to be 
$$
y^{\langle\lambda,D\rangle}=\prod_{\sigma\in \Sigma(1)}y_{\sigma}^{\langle\varpi_{\sigma},\lambda\rangle+n_{\sigma}}.
$$
One verifies that it induces an isomorphism
$$
A[D]\cong \bigoplus_{\lambda\in X^{*}(\wht)\cap P_{D}} \bc\cdot \lambda.
$$

Conversely, given a full dimensional lattice polytope $P$ in $\kt$, one can construct its normal fan $\Sigma_{P}$ in $\kt^{*}$ and a $\wht$-invariant ample divisor $D_{P}$ on $Y_{\Sigma_{P}}$. In this way, we have a bijection between the set of full dimensional lattice polytopes $P$ in $\kt$ and the set of the pairs $(Y_{\Sigma_{P}},D_{P})$, where $Y_{\Sigma_{P}}$ is a complete toric variety compactifying $\wht$ and $D_{P}$ is a $\wht$-invariant divisor on $Y_{\Sigma_{P}}$.

\begin{prop}
Let $D$ be a $\wht$-invariant ample divisor on $Y$, then
$$
H^{i}(Y,\co_{Y}(D))=0,\quad \forall\,i\neq 0.
$$
\end{prop}




\subsection{Geometric interpretation}

For any $P\in \cf(T)$, we have the strongly convex rational polyhedral cone in $\ka_{T}^{G,*}=\kt^{*}$:
$$
a_{P}^{G,+}:=\{\chi\in \ka_{T}^{G,*}\mid \forall \alpha\in \Phi(M_{P},T),\,\chi(\alpha^{\vee})=0;\, \forall \beta\in \Phi(N_{P},T),\,\chi(\beta^{\vee})\geqslant 0\}.
$$
Then $\Sigma=(a_{P}^{G,+})_{P\in \cf(T)}$ is a complete fan in $\kt$. (This is the reason why we work with $\pgl_{d+1}$ instead of $\gl_{d+1}$). Let $Y=Y_{\Sigma}$ be the toric compactification of $\widehat{T}$ according to the fans $\Sigma$. Since we have supposed that $\bnn\in \bn^{d}$, the normal fan of the polytope $E_{\bnn}\subset \ka_{T}^{G}=\kt$ is $\Sigma$. So $E_{\bnn}$ defines an ample divisor $D_{\bnn}$ on $Y$: Suppose that 
\begin{equation}\label{definepoly}
E_{\bnn}=\{a\in \kt\,|\, \langle \varpi_{\sigma},a\rangle +M_{\bnn,\sigma}\geq 0,\;\forall\, \sigma\in \Sigma(1)\},
\end{equation}
then $D_{\bnn}$ is defined to be $\sum_{\sigma\in \Sigma(1)}M_{\bnn,\sigma}\cdot D_{\sigma}$. According to the toric dictionary, we have 
$$
h^{0}(D_{\bnn}):=\dim\big( H^{0}(Y,\co_{Y}(D_{\bnn}))\big)=|E_{\bnn}\cap X_{*}(T)|=|\Lambda_{\bnn}|.
$$

Our main tool in this section is the Hirzebruch-Riemann-Roch theorem.

\begin{thm}[Hirzebruch-Riemann-Roch]

Let $X$ be a proper $k$-scheme. There exists a Todd class $\Td(X)$ of $X$ such that for any vector bundle $E$ on $X$, we have
$$
\chi(X, E)=\int_{X} \ch(E)\cap \Td(X), 
$$
where $\ch(E)$ is the Chern character of $E$.

\end{thm}

\begin{prop}

The power series $Q_{2}(q;\vec{t}\;)$ is a rational fraction.

\end{prop}

\begin{proof}

It is easy to see that
\begin{eqnarray*}
Q_{2}(q;\vec{t}\;)&=& \sum_{\bnn\in \bn^{d}}\sum_{q=0}^{+\infty}\dim\left(\qlbar^{\Lambda_{\bnn}}\otimes \bs_{m}\right)q^{m}t^{\bnn}\\
&=&(1-q)^{-d}\sum_{\bnn\in \bn^{d}}|\Lambda_{\bnn}|\cdot t^{\bnn}\\
&=:& (1-q)^{-d}Q_{4}(\vec{t}\;),
\end{eqnarray*}
where $t^{\bnn}=t_{1}^{n_{1}}\cdots t_{d}^{n_{d}}$. So it is enough to prove that $Q_{4}(\vec{t}\;)$ is a rational fraction. We can further regroup the summation in $\bnn\in \bn^{d}$ according to the ordered partition $\bar{d}=(\{i_{1},\cdots, i_{j_{1}}\}; \cdots;\{i_{j_{r-1}+1},\cdots, i_{j_{r}}\})$ of the set $\{1,\cdots, d\}$, written $\bar{d}\vdash \{1,\cdots, d\}$, i.e.
$$
\sum_{\bnn\in \bn^{d}}=\sum_{\bar{d}\,\vdash \{1,\cdots, d\}}\sum_{n_{i_{1}}=\cdots=n_{i_{j_{1}}}\\<\cdots< \\n_{i_{j_{r-1}}+1}=\cdots =n_{i_{j_{r}}}},
$$
we can write
$$
Q_{4}(\vec{t}\;)=\sum_{\bar{d}\,\vdash \{1,\cdots, d\}}Q_{4,\bar{d}}(\vec{t}\;),
$$
and it is sufficient to prove that each $Q_{4,\bar{d}}(\vec{t}\;)$ is a rational fraction, since there are only finitely many ordered partitions.

\begin{lem}\label{linearface}

For each ordered partition $\bar{d}\vdash \{1,\cdots, d\}$ and for each $\sigma\in \Sigma(1)$, there is a linear form $L_{\bar{d},\sigma}:\bz^{d}\to \bz$ such that in the defining equation (\ref{definepoly}) of $E_{\bnn}$, we have 
$$
M_{\bnn,\sigma}=L_{\bar{d},\sigma}(\bnn)
$$ 
for all the $\bnn\in \bn^{d}$ which are of type $\bar{d}$.

\end{lem}

\begin{proof}

Let $P_{\sigma}\in \cf(T)$ be the parabolic subgroup of $G$ corresponding to $\sigma$, let $B\in \cp(T)$ be any Borel subgroup contained in $P_{\sigma}$, then we have
$$
M_{\bnn,\sigma}=-\varpi_{\sigma}(H_{B}(x_{0})).
$$
Take a minimal gallery $B_{1},\cdots, B_{r}\in \cp(T)$ such that $B_{1}=B_{0}^{-}$ and $B_{r}=B$, and $B_{i},B_{i+1}$ are adjacent Borel subgroups, $i=1,\cdots,r-1$. By proposition \ref{gkmreg}, we have
$$
H_{B_{i+1}}(x_{0})-H_{B_{i}}(x_{0})=\val(\beta_{i}(\gamma))\cdot \beta_{i}^{\vee},\quad i=1,\cdots, r-1,
$$
where $\beta_{i}$ is the unique root which is positive with respect to $B_{i+1}$ while negative with respect to $B_{i}$. By the definition of root valuation, 
$$
\val(\beta_{i}(\gamma))=n_{j(i)},
$$
for some $j(i)\in \{1,\cdots, d\}$, depending on the type $\bar{d}$ of $\bnn$. So
\begin{eqnarray*}
M_{\bnn,\sigma}&=&-\varpi_{\sigma}(H_{B}(x_{0}))=-\sum_{i=1}^{r-1}\varpi_{\sigma} \big(H_{B_{i+1}}(x_{0})-H_{B_{i}}(x_{0})\big)+\varpi_{\sigma} (H_{B_{1}}(x_{0}))\\
&=& -\sum_{i=1}^{r-1}\varpi_{\sigma} \big(n_{j(i)}\cdot \beta_{i}^{\vee}\big)
\end{eqnarray*}
depends linearly on $\bnn$ according to its type $\bar{d}$.

\end{proof}

Since $D_{\bnn}$ is an ample divisor, we have $H^{i}(Y,\co_{Y}(D_{\bnn}))=0,\,\forall i\neq 0$. By the Hirzebruch-Riemann-Roch theorem, we have
$$
|\Lambda_{\bnn}|=h^{0}(D_{\bnn})=\chi(Y,\co_{Y}(D_{\bnn}))=\int_{Y}\ch(\co_{Y}(D_{\bnn}))\cap \Td(Y).
$$ 

Since $D_{\bnn}=\sum_{\sigma\in \Sigma(1)}M_{\bnn,\sigma}\cdot D_{\sigma}$ and $M_{\bnn,\sigma}$ is linear for all the $\bnn$ of the fixed type $\bar{d}$ by lemma \ref{linearface}, there exists a polynomial $Q_{5,\bar{d}}\in \bq[T_{1},\cdots, T_{d}]$ such that 
$$
|\Lambda_{\bnn}|=Q_{5,\bar{d}}(\bnn),
$$
for all $\bnn$ of type $\bar{d}$. This implies that 
$$
Q_{4,\bar{d}}(\vec{t}\;)=\sum_{\bnn \text{ of type }\bar{d}}Q_{5,\bar{d}}(\bnn)t^{\bnn}
$$ 
is a rational fraction, which concludes the proof.

\end{proof}

As a consequence, if we assume the cohomological purity of $F_{\gamma}$, the conjecture \ref{rationality} is equivalent to 

\begin{conj}

The power series $Q_{1}(q;\vec{t}\;)$ is a rational fraction.

\end{conj}

We will give a geometric interpretation of the module 
$
R_{\bnn}=\sum_{\alpha\in \Phi^{+}}\sum_{i=1}^{\val(\alpha(\gamma))}R_{\alpha,i},
$
or more precisely, of the homogeneous degree $l$ part of
$$
R_{\alpha, i}=\sum_{\lambda \text{ satisfying }(*)} (1-\alpha^{\vee})^{i}\lambda\otimes \bs\{\partial_{\alpha}^{i}\}\subset \qlbar^{\Lambda_{\bnn}} \otimes \bs,
$$
where $(*)$ refers to the condition:
$$
\lambda, \alpha^{\vee}\lambda,\cdots, (\alpha^{\vee})^{i}\lambda \in \Lambda_{\bnn}.
$$
As in the equation (\ref{lambda function}), the coroot $\alpha^{\vee}$ defines the meromorphic function on $Y$
$$
y^{\alpha^{\vee}}=\prod_{\sigma\in \Sigma(1)}y_{\sigma}^{\langle\varpi_{\sigma},\alpha^{\vee}\rangle}.
$$

Let 
$$
y^{\alpha^{\vee}}_{+}=\prod_{\substack{\sigma\in \Sigma(1)\\ \langle\varpi_{\sigma},\alpha^{\vee}\rangle\geq 0}}y_{\sigma}^{\langle\varpi_{\sigma},\alpha^{\vee}\rangle},\quad y^{\alpha^{\vee}}_{-}=\prod_{\substack{\sigma\in \Sigma(1)\\ \langle\varpi_{\sigma},\alpha^{\vee}\rangle< 0}}y_{\sigma}^{-\langle\varpi_{\sigma},\alpha^{\vee}\rangle}.$$

Let $D_{\alpha}$ be the divisor on $Y$ defined by the homogeneous polynomial
$
y^{\alpha^{\vee}}_{+}-y^{\alpha^{\vee}}_{-}.
$ 
After homogenisation, it is easy to see

\begin{prop}
We have the identity
$$
\sum_{\lambda \text{ satisfying }(*)} (1-\alpha^{\vee})^{i}\lambda
=H^{0}(Y,\co_{Y}(D_{\bnn}-iD_{\alpha})).
$$
\end{prop} 

To interpret $\bs\{\partial_{\alpha}^{i}\}$, we use the fact that $\bs$ is the homogeneous coordinate ring of $\bp(\kt^{*})=\bp^{d-1}$. (Here we suppose that $d\geq 3$, this is not an essential constraint since we will prove the rationality conjecture for $\gl_{2}$ and $\gl_{3}$ by direct calculations.) The vector $\partial_{\alpha}$ in $\kt^{*}$ defines a point $p_{\alpha}\in \bp(\kt^{*})$ which represents the line $k\cdot \partial_{\alpha}$ in $\kt^{*}$. Let $\km_{\alpha}$ be the defining ideal of $p_{\alpha}$ in $\bp(\kt^{*})$.

\begin{lem}[Ensalem-Iarrobino\cite{ei}]

We have $\bs_{l}\{\partial_{\alpha}^{i}\}=(\km_{\alpha}^{l-i})_{l}$, with the subscript $l$ refers to the homogeneous degree $l$ parts. Let $\tp_{\alpha}$ be the blow-up of $\bp(\kt^{*})$ at $p_{\alpha}$, let $E_{\alpha}$ be the resulting exceptional divisor, then
$$
\bs_{l}\{\partial_{\alpha}^{i}\}=H^{0}(\tp_{\alpha},\co(lH+(i-l)E_{\alpha})),
$$
where $H$ is the pull back of a general hyperplane in $\bp(\kt^{*})$. 
\end{lem}

\begin{proof}

For the first assertion, we complete $\partial_{\alpha}$ into a basis of $\kt^{*}$ and write the resulting coordinate to be $(X_{0},\cdots, X_{n})$. It is easy to see that $\bs_{l}\{\partial_{\alpha}^{i}\}$ is generated by the monomials
$
X_{0}^{d_{0}}\cdots X_{n}^{d_{n}},
$
with $\sum d_{i}=l$ and $d_{0}\leq i$. But this is exactly $(\km_{\alpha}^{l-i})_{l}$. It is evident that the second assertion is a reformulation of the first one.

\end{proof}

\begin{cor}
We have the equality
$$
(R_{\bnn})_{l}=\sum_{\alpha\in \Phi^{+}}\sum_{i=1}^{\val(\alpha(\gamma))} H^{0}(Y,\co_{Y}(D_{\bnn}-iD_{\alpha}))\otimes H^{0}(\tp_{\alpha},\co(lH+(i-l)E_{\alpha})).
$$

\end{cor}

\section{The generating series for $\gl_{2}$}

Let $G=\gl_{2}$. Any element $\gamma\in \kt(\co)$ is automatically in minimal form, let $n$ be the root valuation of $\gamma$. Let $F_{\gamma}$ be the fundamental domain of $\xx_{\gamma}$ containing $x_{0}$, where $x_{0}$ is the Kostant regular point defined in \S2.3. By proposition \ref{classical example}, $\ec(x_{0})$ is the interval in $\kt$ between $(0,n),\,(n,0)\in X_{*}(T)$. The fundamental domain is thus the intersection
$$
F_{\gamma}=\sch(n,0)\cap \xx_{\gamma}.
$$

\begin{prop}

We have $F_{\gamma}=\sch(n,0)$, so it admits an affine paving. Its Poincar\'e polynomial is $\sum_{i=0}^{n}q^{n}$.
 
\end{prop}

\begin{proof}

For any $a\in k$, it is evident that for any lattice $L$ in $F^{2}$, we have
$$
\gamma\cdot L\subset L\Longleftrightarrow (a\ep^{n}\mathrm{Id} +\gamma)\cdot L\subset L.
$$
So without any loss of generality, we can assume that both of the eigenvalues of $\gamma$ have valuation $n$. Now that $\sch(n,0)$ parametrise the lattices $L$ of index $n$ satisfying
$$
\kp^{n}\oplus \kp^{n}  \subset L\subset \co\oplus \co,
$$
we have 
$$
\gamma\cdot L\subset \kp^{n}\oplus \kp^{n} \subset L,
$$
which implies that $\sch(n,0)\subset F_{\gamma}$, hence the equality in the first assertion.

Let $I$ be the standard Iwahori subgroup, i.e. it is the inverse image of the Borel subgroup $B_{0}$ under the reduction $G(\co)\rightarrow G(k)$. Recall that we have the Bruhat-Tits decomposition
$$
\sch(n,0)=\bigsqcup_{\mu\in \sch(n,0)^{T}}I\ep^{\mu}K/K,
$$
from which we get the second assertion in the proposition.
\end{proof}

\begin{cor}
For the group $\gl_{2}$, we have 
$$
Q(q;\vec{t}\;)=\frac{1}{q-1}\left(\frac{q^{2}t}{1-qt}-\frac{t}{1-t} \right).
$$

\end{cor}

\section{The generating series for $\gl_{3}$}

Let $G=\gl_{3}$, let $\bnn=(n_{1},n_{2})\in \bn^{2},\, n_{1}\leq n_{2}$, let $\gamma\in \kt(\co)$ be in minimal form with root valuation $\bnn$. Let $F_{\gamma}$ be the fundamental domain of $\xx_{\gamma}$ containing $x_{0}$, where $x_{0}$ is the Kostant regular point defined in \S2.3.

\begin{prop}\label{intersection}
The fundamental domain $F_{\gamma}$ is the intersection of $\xx_{\gamma}$ with 
$$
\sch(2n_{1}+n_{2},0,0)\cap \diag(\ep^{-n_{2}},\ep^{-n_{1}},\ep^{-n_{1}})\cdot\sch(2n_{1}+n_{2},2n_{1}+n_{2},0).
$$
\end{prop}

\begin{proof}

By proposition \ref{classical example}, $\ec(x_{0})$ is the hexagon with vertices marked as indicated in figure \ref{truncation1}.

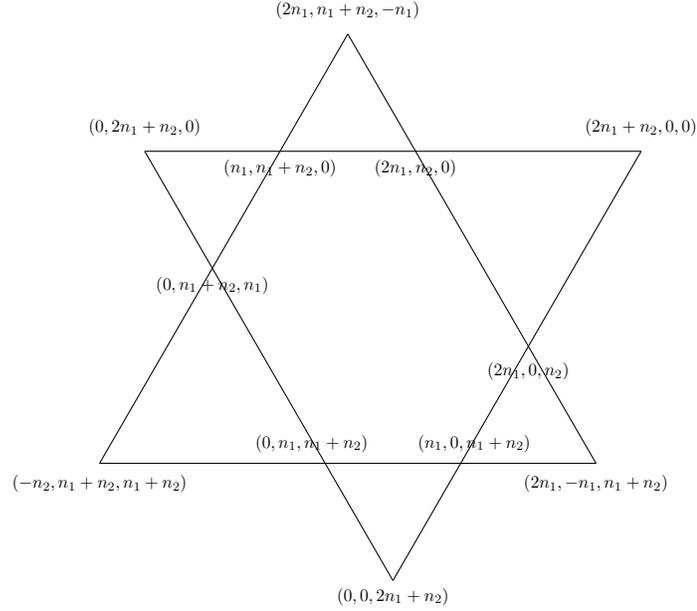
\begin{figure}[h]
\begin{center}
\begin{tikzpicture}[node distance = 2cm, auto, scale=0.6, transform shape]
\draw (-5.5, 3.17)--(5.5,3.17);
\draw (-5.5, 3.17)--(0,-6.35);
\draw (0,-6.35)--(5.5,3.17);
\draw (-1, 5.77)--(-6.5,-3.75);
\draw (-6.5, -3.75)--(4.5, -3.75);
\draw (4.5,-3.75)--(-1,5.77);

\node at (-1,6.3) {$(2n_{1},n_{1}+n_{2},-n_{1})$};
\node at (-6.5,-4.2) {$(-n_{2},n_{1}+n_{2},n_{1}+n_{2})$};
\node at (4.5,-4.2) {$(2n_{1},-n_{1},n_{1}+n_{2})$};

\node at (-5.5, 3.7) {$(0,2n_{1}+n_{2},0)$};
\node at (5.5, 3.7) {$(2n_{1}+n_{2},0,0)$};
\node at (0, -6.7){$(0,0, 2n_{1}+n_{2})$};

\node at (0.5, 2.8) {$(2n_{1},n_{2},0)$};
\node at (-2.5,2.8) {$(n_{1},n_{1}+n_{2},0)$};
\node at (-4,0.2) {$(0,n_{1}+n_{2}, n_{1})$};
\node at (-1.8, -3.3){$(0,n_{1},n_{1}+n_{2})$};
\node at (1.8, -3.3){$(n_{1},0,n_{1}+n_{2})$};
\node at (3, -1.7){$(2n_{1},0, n_{2})$};

\end{tikzpicture}
\caption{Hexagon as intersection of two triangles.}
\label{truncation1}
\end{center}
\end{figure}

This hexagon can also be represented as the intersection of two triangles as indicated also in the figure. Let $\vartriangle,\,\triangledown$ be the upward and the downward triangle in the figure. According to the example \ref{tschubert}, we see that $\ec(x)\in \triangledown$ if and only if $x\in \sch(2n_{1}+n_{2},0,0)$. We notice that $\vartriangle$ is the translation by $(-n_{2},-n_{1},-n_{1})$ of the triangle $\vartriangle'$ with vertices
$$
(2n_{1}+n_{2},2n_{1}+n_{2},0),\quad (2n_{1}+n_{2},0,2n_{1}+n_{2}),\quad (0,2n_{1}+n_{2},2n_{1}+n_{2}).
$$
Again using example \ref{tschubert}, we see that $\ec(x)\in \vartriangle'$ if and only if $x\in \sch(2n_{1}+n_{2},2n_{1}+n_{2},0)$, the result follows directly from these considerations.

\end{proof}

\subsection{Affine paving}

We can pave $F_{\gamma}$ in affine spaces, the strategy is the following: By proposition \ref{intersection}, we have 
$$
F_{\gamma}=\xx_{\gamma}\cap\sch(2n_{1}+n_{2},0,0)\cap \diag(\ep^{-n_{2}},\ep^{-n_{1}},\ep^{-n_{1}})\cdot\sch(2n_{1}+n_{2},2n_{1}+n_{2},0).
$$
So we firstly pave the intersection of the two affine schubert varieties in affine spaces, but this paving doesn't induce an affine paving of $F_{\gamma}$, we need to regroup the resulting pavements and do a second nonstandard paving. Our main result in this section is:

\begin{thm}
The fundamental domain $F_{\gamma}$ admits an affine paving, which only depends on the root valuation of $\gamma$.
\end{thm}

\begin{proof}

Let $I$ be the standard Iwahori subgroup, i.e. it is the inverse image of the Borel subgroup $B_{0}$ under the reduction $G(\co)\rightarrow G(k)$. Let $I'=\Ad(\diag(\ep^{n_{1}},\ep^{n_{2}},\ep^{n_{2}}))I$. By \cite{chen} corollary 2.3, we have the affine paving
\begin{eqnarray*}
&&\sch(2n_{1}+n_{2},0,0)\cap \diag(\ep^{-n_{2}},\ep^{-n_{1}},\ep^{-n_{1}})\cdot\sch(2n_{1}+n_{2},2n_{1}+n_{2},0)\\
&&=\bigsqcup_{\mu\in (F_{\gamma})^{T}}\sch(2n_{1}+n_{2},0,0)\cap I'\ep^{\mu}K/K\\
&&=\bigsqcup_{\mu\in (F_{\gamma})^{T}}\begin{bmatrix}
\co&\kp^{a}&\kp^{b}\\
\kp^{n_{2}-n_{1}+1}&\co&\co\\
\kp^{n_{2}-n_{1}+1}&\kp&\co
\end{bmatrix}\ep^{\mu}K/K,
\end{eqnarray*}
where $a=\max\{n_{1}-n_{2},-\mu_{2}\},\,b=\max\{n_{1}-n_{2},-\mu_{3}\}$. We denote by $C(\mu)$ the resulting pavement containing $\ep^{\mu}$.

To pave $F_{\gamma}$, we cut it into $4$ parts. Let $(\mu'_{1},\mu'_{2},\mu'_{3})=(\mu_{1}-n_{1},\mu_{2}-n_{2},\mu_{3}-n_{2})$, and
\begin{eqnarray*}
R_{1}&=&\{\mu\in (F_{\gamma})^{T}\mid \mu'_{1}\leq \mu'_{2},\mu'_{3}\},\\
R'_{1}&=&\{\mu\in (F_{\gamma})^{T}\mid \mu'_{1}\geq \mu'_{2},\mu'_{3};\,\mu_{2}\leq n_{2}-n_{1};\,\mu_{3}\leq n_{2}-n_{1}\},\\
R_{2}&=&\{\mu\in (F_{\gamma})^{T}\mid \mu'_{2}< \mu'_{1},\mu'_{3};\,\mu_{3}>n_{2}-n_{1}\},\\
R_{3}&=&\{\mu\in (F_{\gamma})^{T}\mid \mu'_{3}< \mu'_{1},\mu'_{2};\,\mu_{2}>n_{2}-n_{1}\}.
\end{eqnarray*}
Although $R_{1}$ and $R'_{1}$ may intersect at one point, it doesn't cause trouble to the paving. Figure \ref{nonstandard} gives an idea of the cutting. Let $V_{i}=\bigsqcup_{\mu\in R_{i}}C(\mu),\,i=1,2,3$. For $l\in \bz$, let $R_{i,l}=\{\mu\in R_{i}\mid\mu_{i}=l\}$ and $V_{i,l}=\bigsqcup_{\mu\in R_{i,l}}C(\mu)$. Similar notations for $R'_{1}$.

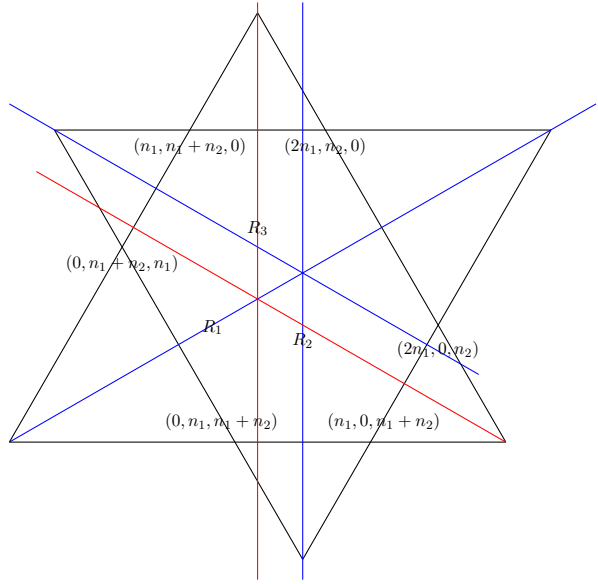
\begin{figure}[h]
\begin{center}
\begin{tikzpicture}[node distance = 2cm, auto, scale=0.6, transform shape]
\draw (-5.5, 3.17)--(5.5,3.17);
\draw (-5.5, 3.17)--(0,-6.35);
\draw (0,-6.35)--(5.5,3.17);
\draw (-1, 5.77)--(-6.5,-3.75);
\draw (-6.5, -3.75)--(4.5, -3.75);
\draw (4.5,-3.75)--(-1,5.77);

\draw [blue] (0,6)--(0,-6.8);
\draw [blue] (-6.5,3.75)--(3.9,-2.25);
\draw [blue] (-6.5,-3.75)--(6.5,3.75);

\draw [red]  (-1, 6)--(-1,-6.8);
\draw [red]  (4.5, -3.75)--(-5.9,2.25);

\node at (0.5, 2.8) {$(2n_{1},n_{2},0)$};
\node at (-2.5,2.8) {$(n_{1},n_{1}+n_{2},0)$};
\node at (-4,0.2) {$(0,n_{1}+n_{2}, n_{1})$};
\node at (-1.8, -3.3){$(0,n_{1},n_{1}+n_{2})$};
\node at (1.8, -3.3){$(n_{1},0,n_{1}+n_{2})$};
\node at (3, -1.7){$(2n_{1},0, n_{2})$};

\node at (-2,-1.2) {$R_{1}$};
\node at (-1, 1){$R_{3}$};
\node at (0,-1.5){$R_{2}$};

\end{tikzpicture}
\caption{Nonstandard paving.}
\label{nonstandard}
\end{center}
\end{figure}

We use the Iwahori subgroup $I'$ to pave $V_{1}\cap \xx_{\gamma}$. Since we have
$$
C(\mu)=\begin{bmatrix}
\co&&\\
\kp^{n_{2}-n_{1}+1}&\co&\co\\
\kp^{n_{2}-n_{1}+1}&\kp&\co
\end{bmatrix}\ep^{\mu}K/K,$$
we see easily that $C(\mu)\cap \xx_{\gamma}$ is isomorphic to an affine space.

We also use $I'$ to pave $V'_{1}\cap \xx_{\gamma}$. We have
$$
C(\mu)=\begin{bmatrix}
\co&\kp^{-\mu_{2}}&\kp^{-\mu_{3}}\\
&\co&\co\\
&\kp&\co
\end{bmatrix}\ep^{\mu}K/K.$$
It is easily checked that $C(\mu)\cap \xx_{\gamma}$ is isomorphic to an affine space.

We need a second nonstandard paving in order to pave $V_{2}\cap \xx_{\gamma}$ and $V_{3}\cap \xx_{\gamma}$. Since they are symmetric, we only give details for $V_{3}\cap \xx_{\gamma}$. Since $V_{3}=\bigsqcup_{l\in \bz} V_{3,l}$, we only need to pave $V_{3,l}\cap \xx_{\gamma}$. Let
$
I_{l}'=\Ad(\diag(1,\ep^{l},\ep^{l}))I',
$
we claim that 
$$
V_{3,l}\cap \xx_{\gamma}=\bigsqcup_{\mu\in R_{3,l}}V_{3,l}\cap I'_{l}\ep^{\mu}K/K \cap \xx_{\gamma}
$$
is an affine paving. Since we have
$$
C(\mu)=\begin{bmatrix}
\co&\kp^{n_{1}-n_{2}}&\kp^{-\mu_{3}}\\
\kp^{n_{2}-n_{1}+1}&\co&\co\\
&\kp&\co
\end{bmatrix}\ep^{\mu}K/K,
$$
we see easily that $V_{3,l}$ admits an affine fibration onto the closed subvariety 
$$
\bigsqcup_{\mu\in R_{3,l}}\begin{bmatrix}
\co&\kp^{n_{1}-n_{2}}&\\
\kp^{n_{2}-n_{1}+1}&\co&\\
&&\co
\end{bmatrix}\ep^{\mu}K/K$$ 
of $\xx^{\gl_{2}\times \gl_{1}}$. This implies 
$$
V_{3,l}\cap I'_{l}\ep^{\mu}K/K=\begin{bmatrix}
\co&\kp^{c}&\kp^{-l}\\
\kp^{n_{2}-n_{1}+l+1}&\co&\co\\
&\kp&\co
\end{bmatrix}\ep^{\mu}K/K,
$$
with $c=\max(n_{1}-n_{2}-l, -\mu_{2})$. With this equality, it is easily checked that $V_{3,l}\cap I'_{l}\ep^{\mu}K/K \cap \xx_{\gamma}$ is isomorphic to an affine space.

It remains to precise the order of the paving. First of all, the Bruhat-Tits order with respect to $I'$ induces an ordering of $V'_{1,l}$ and $V_{i,l},\,i=1,2,3,\,l\in \bz$. On $V_{1,l}$ and $V'_{1, l}$ we use the Bruhat-Tits order with respect to $I'$, while on $V_{2,l}$ and $V_{3,l}$, we use the Bruhat-Tits order with respect to $I'_{l}$.

\end{proof}

\subsection{Rationality conjecture}

To calculate the Poincaré polynomial of $F_{\gamma}$, we proceed by an indirect way in order to avoid the combinatorial complexity. Our strategy is the following: we calculate firstly the Poincaré polynomial of $\xx_{\gamma}\cap \sch(2n_{1}+n_{2}, 0,0)$, then we calculate the Poincaré polynomial of the complement of $F_{\gamma}$ in it, their difference gives what we want. It turns out that the complement of $F_{\gamma}$ can be paved in affine spaces.

\begin{thm}\label{main2}

The Poincaré polynomial of $F_{\gamma}$ is
$$
P_{\bnn}(t)=\sum_{i=1}^{n_{1}}i(t^{4i-2}+t^{4i-4})+\sum_{i=2n_{1}}^{n_{1}+n_{2}-1}(2n_{1}+1)t^{2i}+\sum_{i=n_{1}+n_{2}}^{2n_{1}+n_{2}-1}4(2n_{1}+n_{2}-i)t^{2i}+t^{4n_{1}+2n_{2}}.
$$
\end{thm}

Taking into account the fact that $F_{n_{2},n_{1}}$ has the same Poincaré polynomial as $F_{n_{1},n_{2}}$, we get the precise expression for the generating series.

\begin{cor}
The power series
$$
\sum_{n_{1}=1}^{+\infty}\sum_{n_{2}=1}^{+\infty}  P_{(n_{1},n_{2})}(t)\, T_{1}^{n_{1}} T_{2}^{n_{2}}\in \bz[t][[T_{1},T_{2}]]
$$
equals the rational fraction
\begin{eqnarray*}
\lefteqn{2\Bigg\{\frac{(t^{2}+1)T_{1}T_{2}}{(1-T_{2})(1-T_{1}T_{2})(1-t^{4}T_{1}T_{2})^{2}}+\frac{t^{4}T_{1}T_{2}^{2}(3-t^{4}T_{1}T_{2})}{(1-T_{2})(1-t^{2}T_{2})(1-t^{4}T_{1}T_{2})^{2}}}\\
&+&\frac{4t^{4}T_{1}T_{2}}{(1-t^{2}T_{2})(1-t^{4}T_{1}T_{2})^{2}(1-t^{6}T_{1}T_{2})}+\frac{t^{6}T_{1}T_{2}}{(1-t^{2}T_{2})(1-t^{6}T_{1}T_{2})}\Bigg\}\\
&-&\left[\frac{(t^{2}+1)T_{1}T_{2}}{(1-T_{1}T_{2})(1-t^{4}T_{1}T_{2})^{2}}+\frac{4t^{4}T_{1}T_{2}}{(1-t^{4}T_{1}T_{2})^{2}(1-t^{6}T_{1}T_{2})}+\frac{t^{6}T_{1}T_{2}}{1-t^{6}T_{1}T_{2}}\right]
\end{eqnarray*}

\end{cor}

\begin{proof}[Proof of the theorem \ref{main2}]

To pave $\xx_{\gamma}\cap\sch(2n_{1}+n_{2},0,0)$, we use the same idea as the proof of theorem 3.11 in \cite{chen}. We can pave $\sch(2n_{1}+n_{2},0,0)$ in affine spaces with the Iwahori subgroup
$$
I'=\Ad(\diag(\ep^{2n_{1}+n_{2}},1,1))I.
$$
Let $C(\mu)=\sch(2n_{1}+n_{2},0,0)\cap I'\ep^{\mu}K/K$, then we have the affine paving
$$
\sch(2n_{1}+n_{2},0,0)=\bigsqcup_{\mu\in \sch(2n_{1}+n_{2},0,0)^{T}}C(\mu),
$$
with
$$
C(\mu)=\begin{bmatrix}
\co&&\\
\kp^{-\mu_{1}}&\co&\co\\
\kp^{-\mu_{1}}&\kp&\co
\end{bmatrix}\ep^{\mu}K/K.
$$
Then we prove with the same method that $C(\mu)\cap \xx_{\gamma}$ is an affine space of dimension
$$
\min\{n_{1}, \mu_{2}\}+\min\{n_{1},\mu_{3}\}+\min\left\{n_{2}, \vert \mu_{2}-\mu_{3}\vert +\frac{\mathrm{sign}(\mu_{2}-\mu_{3})-1}{2}\right\}.
$$
It suffices to count the number of affine pavements of each dimension to get the Poincaré polynomial. To facilitate the work, we cut $\sch(2n_{1}+n_{2},0,0)^{T}$ into $7$ parts, as indicated in  figure \ref{partition}, where

\begin{eqnarray*}
R_{1}&=&\{\mu\in \sch(2n_{1}+n_{2},0,0)^{T}\mid \mu_{2}-\mu_{3}>n_{2}\},\\
R'_{1}&=&\{\mu\in \sch(2n_{1}+n_{2},0,0)^{T}\mid \mu_{3}-\mu_{2}>n_{2}\},\\
R_{2}&=&\{\mu\in \sch(2n_{1}+n_{2},0,0)^{T}\mid \mu_{2}-\mu_{3}\leq n_{2}, \mu_{3}<n_{1},\mu_{2}>n_{1}\},\\
R'_{2}&=&\{\mu\in \sch(2n_{1}+n_{2},0,0)^{T}\mid \mu_{3}-\mu_{2}\leq n_{2}, \mu_{2}<n_{1},\mu_{3}>n_{1}\},\\
R_{3}&=&\{\mu\in \sch(2n_{1}+n_{2},0,0)^{T}\mid  \mu_{3}\geq n_{1},\mu_{2}\geq n_{1}\},\\
R_{4}&=&\{\mu\in \sch(2n_{1}+n_{2},0,0)^{T}\mid  \mu_{3}\leq n_{1},\mu_{2}\leq n_{1}, n_{2}<\mu_{1}\leq n_{1}+n_{2}\},\\
R'_{4}&=&\{\mu\in \sch(2n_{1}+n_{2},0,0)^{T}\mid  \mu_{3}< n_{1},\mu_{2}< n_{1}, n_{1}+n_{2}<\mu_{1}\leq 2n_{1}+n_{2}\},\\
\end{eqnarray*}

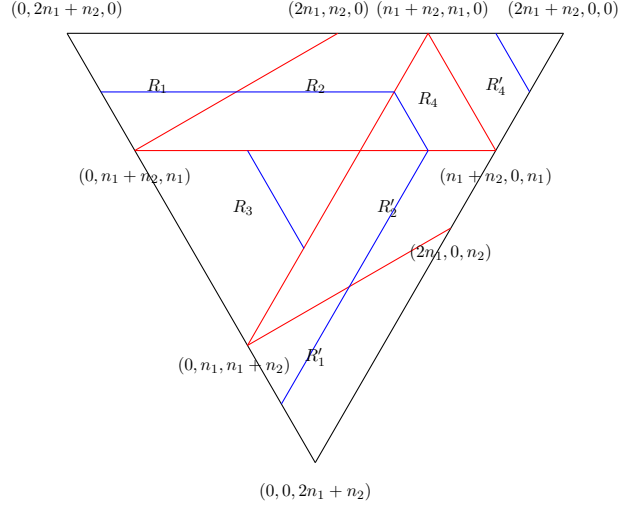
\begin{figure}[h]
\begin{center}
\begin{tikzpicture}[node distance = 2cm, auto, scale=0.6, transform shape]
\draw (-5.5, 3.17)--(5.5,3.17);
\draw (-5.5, 3.17)--(0,-6.35);
\draw (0,-6.35)--(5.5,3.17);

\draw [red] (-4,0.57)--(0.5,3.17);
\draw [red] (-1.5,-3.75)--(3,-1.15);
\draw [red] (-4,0.57)--(4,0.57);
\draw [red] (-1.5,-3.75)--(2.5,3.17);
\draw [red] (4,0.57)--(2.5,3.17);

\draw [blue] (-4.75,1.87)--(1.75,1.87);
\draw [blue] (1.75,1.87)--(2.5, 0.57);
\draw [blue] (-0.75,-5.05)--(2.5, 0.57);
\draw [blue] (-1.5,0.57)--(-0.25, -1.6);
\draw [blue] (4,3.17)--(4.75,1.87);

\node at (-5.5, 3.7) {$(0,2n_{1}+n_{2},0)$};
\node at (5.5, 3.7) {$(2n_{1}+n_{2},0,0)$};
\node at (0.3, 3.7) {$(2n_{1},n_{2},0)$};
\node at (2.6, 3.7) {$(n_{1}+n_{2},n_{1},0)$};
\node at (4,0) {$(n_{1}+n_{2},0,n_{1})$};
\node at (-4,0) {$(0,n_{1}+n_{2}, n_{1})$};
\node at (-1.8, -4.2){$(0,n_{1},n_{1}+n_{2})$};
\node at (3, -1.7){$(2n_{1},0, n_{2})$};
\node at (0, -7){$(0,0, 2n_{1}+n_{2})$};

\node at (-3.5, 2) {$R_{1}$};
\node at (0, -4) {$R'_{1}$};
\node at (0,2) {$R_{2}$};
\node at (1.6,-0.7) {$R'_{2}$};
\node at (-1.6,-0.7) {$R_{3}$};
\node at (2.5,1.7) {$R_{4}$};
\node at (4,2) {$R'_{4}$};

\end{tikzpicture}
\caption{Partition of the triangle.}
\label{partition}
\end{center}
\end{figure}

Now we count the contribution of each part. We first sum over each blue lines as indicated in figure \ref{partition}, then we add up all blue lines. We use the notation $\sum_{\mu=\nu}^{\nu'}$ to mean summation over the line having ends in $\nu,\,\nu'$. Since the Poincaré polynomial is a polynomial in $t^{2}$, we use $q:=t^{2}$ to simplify the notation.

\begin{enumerate}

\item The contribution of $C(\mu)\cap \xx_{\gamma}, \mu\in R_{1}$ to the Poincaré polynomial of $\xx_{\gamma}\cap\sch(2n_{1}+n_{2},0,0)$ is
$$
\sum_{i=0}^{n_{1}-1}\sum_{\mu=(0,2n_{1}+n_{2}-i,i)}^{(2n_{1}-2i-1,n_{2}+i+1,i)}q^{n_{1}+n_{2}+i}=2\sum_{i=1}^{n_{1}}iq^{2n_{1}+n_{2}-i}.
$$

\item The contribution of $C(\mu)\cap \xx_{\gamma}, \mu\in R'_{1}$ is the same as $R_{1}$.

\item The contribution of $C(\mu)\cap \xx_{\gamma}, \mu\in R_{2}$ is
\begin{eqnarray*}
&&\sum_{i=0}^{n_{1}-1}\sum_{\mu=(2n_{1}-2i,n_{2}+i,i)}^{(n_{1}+n_{2}-i-1, n_{1}+1, i)}q^{i+n_{1}+\mu_{2}-\mu_{3}}\\
&=&n_{1}\sum_{i=2n_{1}+1}^{n_{1}+n_{2}}q^{i}+\sum_{i=1}^{n_{1}-1}(n_{1}-i)q^{n_{1}+n_{2}+i}.
\end{eqnarray*}

\item The contribution of $R'_{2}$ is
\begin{eqnarray*}
&&\sum_{i=0}^{n_{1}-1}\sum_{\mu=(2n_{1}-2i,i, n_{2}+i)}^{(n_{1}+n_{2}-i-1,i, n_{1}+1)}q^{i+n_{1}+\mu_{3}-\mu_{2}-1}\\
&=&n_{1}\sum_{i=2n_{1}+1}^{n_{1}+n_{2}}q^{i-1}+\sum_{i=1}^{n_{1}-1}(n_{1}-i)q^{n_{1}+n_{2}+i-1}.
\end{eqnarray*}

\item The contribution of $R_{3}$ is
\begin{eqnarray*}
&&\sum_{i=0}^{n_{2}}\sum_{\mu=(i,n_{1},n_{1}+n_{2}-i)}^{(i,n_{1}+n_{2}-i, n_{1})}q^{2n_{1}+\vert \mu_{2}-\mu_{3}\vert +\frac{\mathrm{sign}(\mu_{2}-\mu_{3})-1}{2}}\\
&=&\sum_{i=0}^{n_{2}}q^{2n_{1}}(1+q+q^{2}+\cdots+q^{n_{2}-i})\\ &=&q^{2n_{1}}\sum_{i=0}^{n_{2}}(n_{2}+1-i)q^{i}.
\end{eqnarray*}

\item The contribution of $R_{4}$ is 
\begin{eqnarray*}
&&\sum_{i=0}^{n_{1}-1}\sum_{\mu=(n_{1}+n_{2}-i,n_{1},i)}^{(n_{1}+n_{2}-i, i, n_{1})}q^{n_{1}+i+\vert \mu_{2}-\mu_{3}\vert +\frac{\mathrm{sign}(\mu_{2}-\mu_{3})-1}{2}}\\
&=&\sum_{i=0}^{n_{1}-1}q^{n_{1}+i}(1+q+\cdots+q^{n_{1}-i})\\
&=&n_{1}q^{2n_{1}}+\sum_{i=0}^{n_{1}-1}(i+1)q^{n_{1}+i}.
\end{eqnarray*}

\item The contribution of $R'_{4}$ is
\begin{eqnarray*}
&&\sum_{i=0}^{n_{1}-1}\sum_{\mu=(2n_{1}+n_{2}-i, 0, i)}^{(2n_{1}+n_{2}-i,i,0)}q^{i+\vert \mu_{2}-\mu_{3}\vert +\frac{\mathrm{sign}(\mu_{2}-\mu_{3})-1}{2}}\\
&=&\sum_{i=0}^{n_{1}-1}q^{i}(1+q+\cdots+q^{i}).
\end{eqnarray*}

\end{enumerate}

The complement of $F_{\gamma}$ in $\xx_{\gamma}\cap \sch(2n_{1}+n_{2},0,0)$ can be paved in affine spaces in the following way: Observe that $F_{\gamma}$ is contained in the intersection $\xx_{\gamma}\cap \sch(n_{1}+n_{2}, n_{1},0)$, whose complement in $\xx_{\gamma}\cap \sch(2n_{1}+n_{2},0,0)$ can be paved in affine spaces using the standard Iwahori subgroup $I$. It suffices to pave the complement of $F_{\gamma}$ in $\xx_{\gamma}\cap \sch(n_{1}+n_{2}, n_{1},0)$, which can be done by using the Iwahori subgroup
$$
I''=\Ad(\diag(\ep^{n_{1}},\ep^{n_{2}},\ep^{n_{2}}))I.
$$

We cut the complement of $F_{\gamma}^{T}$ in $\sch(2n_{1}+n_{2},0,0)^{T}$ as indicated in figure \ref{complementary}, where
\begin{eqnarray*}
T_{1}&=&\{\mu\in \sch(2n_{1}+n_{2},0,0)^{T}\mid \mu_{1}\geq n_{1}+n_{2}+1\},\\
T_{2}&=&\{\mu\in \sch(2n_{1}+n_{2},0,0)^{T}\mid \mu_{2}\geq n_{1}+n_{2}+1\},\\
T_{3}&=&\{\mu\in \sch(2n_{1}+n_{2},0,0)^{T}\mid \mu_{3}\geq n_{1}+n_{2}+1\},\\
T'_{1}&=&\{\mu\in \sch(2n_{1}+n_{2},0,0)^{T}\mid 2n_{1}+1\leq \mu_{1}\leq n_{1}+n_{2}\}.
\end{eqnarray*}

\begin{figure}[h]
\begin{center}
\begin{tikzpicture}[node distance = 2cm, auto, scale=0.6, transform shape]
\draw (-5.5, 3.17)--(5.5,3.17);
\draw (-5.5, 3.17)--(0,-6.35);
\draw (0,-6.35)--(5.5,3.17);

\draw [red] (-4,0.57)--(-2.5,3.17);
\draw [red] (-1.5,-3.75)--(1.5,-3.75);
\draw [red] (4,0.57)--(2.5,3.17);
\draw [red] (0.5,3.17)--(3,-1.15);

\draw [blue] (-4.75,1.87)--(-3.2,1.87);
\draw [blue] (-0.75,-5.05)--(0, -3.75);
\draw [blue] (4,3.17)--(4.75,1.87);

\node at (-5.5, 3.7) {$(0,2n_{1}+n_{2},0)$};
\node at (5.5, 3.7) {$(2n_{1}+n_{2},0,0)$};
\node at (0.3, 3.7) {$(2n_{1},n_{2},0)$};
\node at (2.6, 3.7) {$(n_{1}+n_{2},n_{1},0)$};
\node at (4,0) {$(n_{1}+n_{2},0,n_{1})$};
\node at (-4,0) {$(0,n_{1}+n_{2}, n_{1})$};
\node at (-1.8, -4.2){$(0,n_{1},n_{1}+n_{2})$};
\node at (3, -1.7){$(2n_{1},0, n_{2})$};
\node at (0, -7){$(0,0, 2n_{1}+n_{2})$};
\node at (-2.5, 3.7){$(n_{1},n_{1}+n_{2},0)$};
\node at (1.5, -4.2){$(n_{1},0,n_{1}+n_{2})$};

\node at (-4, 2.2) {$T_{2}$};
\node at (0, -4.8) {$T_{3}$};
\node at (2.5,1.7) {$T'_{1}$};
\node at (4,2.2) {$T_{1}$};

\end{tikzpicture}
\caption{Complementary of $F_{\gamma}$.}
\label{complementary}
\end{center}
\end{figure}
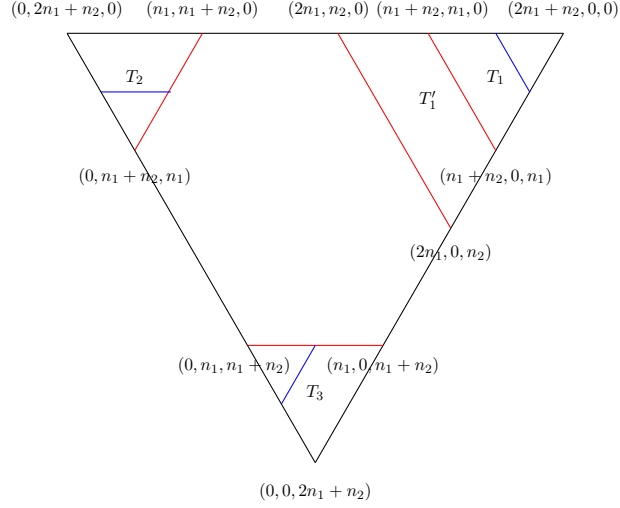

The complement of $\xx_{\gamma}\cap \sch(n_{1}+n_{2}, n_{1},0)$ in $\xx_{\gamma}\cap \sch(2n_{1}+n_{2},0,0)$ is
$$
\bigsqcup_{\mu\in T_{1}\cup T_{2}\cup T_{3}}I\ep^{\mu}K/K\cap \xx_{\gamma}.
$$
It is easy to verify that this is an affine paving. To calculate its Poincaré polynomial, in each region we first sum over the vertices on the blue lines as indicated in figure \ref{complementary}, then we sum over all the lines.

\begin{enumerate}
\item The contribution of $T_{1}$ is
\begin{eqnarray*}
&&\sum_{i=0}^{n_{1}-1}\sum_{\mu=(2n_{1}+n_{2}-i, 0, i)}^{(2n_{1}+n_{2}-i,i,0)}q^{2n_{1}+\vert \mu_{2}-\mu_{3}\vert +\frac{\mathrm{sign}(\mu_{2}-\mu_{3})-1}{2}}\\
&=&\sum_{i=0}^{n_{1}-1}q^{2n_{1}}(1+q+\cdots+q^{i}).
\end{eqnarray*}

\item
The contribution of $T_{2}$ is
$$
\sum_{i=0}^{n_{1}-1}\sum_{\mu=(0,2n_{1}+n_{2}-i,i)}^{(n_{1}-i-1,n_{1}+n_{2}+1,i)}q^{n_{1}+n_{2}+i}=\sum_{i=1}^{n_{1}}iq^{2n_{1}+n_{2}-i}.
$$

\item The contribution of $T_{3}$ is the same as that of $T_{2}$.
\end{enumerate}

It remains to calculate the Poincaré polynomial of the complement of $F_{\gamma}$ in $\xx_{\gamma}\cap \sch(n_{1}+n_{2}, n_{1},0)$. By proposition \ref{intersection}, it is the union
$$
\bigsqcup_{\mu\in T'_{1}}I''\ep^{\mu}K/K\cap\sch(n_{1}+n_{2}, n_{1},0)\cap\xx_{\gamma}.
$$  
By proposition \ref{gkmreg}, points in $I\ep^{\mu}K/K\cap \xx_{\gamma},\,\mu\in T_{2}\cup T_{3}$ don't belong to any $B'\ep^{\nu}K/K\cap \xx_{\gamma},\,\nu\in T'_{1}$, for any $B'\in \cf(T)$. The above intersection is thus equal to 
$$
\bigsqcup_{\mu\in T'_{1}}I''\ep^{\mu}K/K\cap\sch(2n_{1}+n_{2}, 0,0)\cap\xx_{\gamma},
$$
which is easily verified to be an affine space of dimension
$$
2n_{1}+\vert \mu_{2}-\mu_{3}\vert +\frac{\mathrm{sign}(\mu_{2}-\mu_{3})-1}{2}, 
$$
using the equality
$$
I''\ep^{\mu}K/K\cap\sch(2n_{1}+n_{2}, 0,0)=\begin{bmatrix}
\co&\kp^{a}&\kp^{b}\\
&\co&\co\\
&\kp&\co
\end{bmatrix}\ep^{\mu}K/K,
$$
where $a=\max\{n_{1}-n_{2},-\mu_{2}\},\,b=\max\{n_{1}-n_{2},-\mu_{3}\}$.

Summing up the contributions of all the pavements in $T'_{1}$ in the order as for the region $T_{1}$, we find the Poincaré polynomial of the complement of $F_{\gamma}$ in $\xx_{\gamma}\cap \sch(n_{1}+n_{2}, n_{1},0)$ to be

\begin{eqnarray*}
&&\sum_{i=n_{1}}^{n_{2}-1}\sum_{\mu=(2n_{1}+n_{2}-i, 0, i)}^{(2n_{1}+n_{2}-i,i,0)}q^{2n_{1}+\vert \mu_{2}-\mu_{3}\vert +\frac{\mathrm{sign}(\mu_{2}-\mu_{3})-1}{2}}\\
&=&\sum_{i=n_{1}}^{n_{2}-1}q^{2n_{1}}(1+q+\cdots+q^{i}).
\end{eqnarray*}

Now taking into account all the above calculations, we get the result as claimed in the theorem.
\end{proof}

\begin{rem}

Observe that in the above proof we actually give an affine paving of the complement of $F_{\gamma}$ in $\xx_{\gamma}\cap \sch(2n_{1}+n_{2},0,0)$, and this paving can also be obtained by the Arthur-Kottwitz reduction. In principle, one can calculate the Poincar\'e polynomial of the fundamental domain of the affine Springer fibers for $\gl_{4}$, using the same method with Arthur-Kottwitz reduction and the affine pavings in \cite{chen}.

\end{rem}

\section*{Acknowledgement} 

We want to thank G\'erard Laumon and Tam\'as Hausel for their interest in this project, and to Bernd Sturmfels for drawing our attention to his work with Xu Zhiqiang on the Sagbi bases of Cox-Nagata rings. Also, we want to thank an anonymous referee for his careful reading and helpful suggestions.

\end{document}